\documentclass[a4paper,11pt,reqno]{amsart}
\usepackage[T1]{fontenc}
\usepackage{lmodern}
\usepackage[utf8]{inputenc}

\usepackage{amsfonts} 	
\usepackage[centertags]{amsmath}	
\usepackage{dsfont}
\usepackage{stmaryrd}	
\usepackage{amssymb}	
\usepackage{amsthm}		
\usepackage{mathdots}
\usepackage{upgreek}  
\usepackage{mathabx}
\usepackage{enumerate}		
\usepackage{tikz}
\usetikzlibrary{arrows.meta}
\usepackage{footnote}		
\usepackage{subcaption}
\usepackage{mathrsfs}
\usepackage{accents} 

\newcommand\thickbar[1]{\accentset{\rule{.4em}{.6pt}}{#1}}
\newcommand{\MC}{\mathrm{MC}}
\newcommand{\prob}[1]{\mathbb{P}\left(#1\right)}
\newcommand{\ind}[1]{\mathds{1}\left\{#1\right\}}
\newcommand{\eqdist}{\overset{\mathrm{d}}{=}}
\DeclareMathOperator{\supp}{support}
\newcommand{\PM}{\mathcal{P}}
\newcommand{\SM}{\mathcal{S}}
\DeclareMathOperator{\proj}{proj}
\newcommand{\rev}{\text{rev}}
\newcommand{\Xoriginal}{\thickbar{X}}
\newcommand{\Yoriginal}{\thickbar{Y}}
\newcommand{\Zoriginal}{\thickbar{Z}}
\newcommand{\Xcopy}{X}
\newcommand{\Ycopy}{Y}
\newcommand{\Zcopy}{Z}
\theoremstyle{plain}
\newtheorem{thm}{Theorem}
\newtheorem{prop}[thm]{Proposition}
\newtheorem{lem}[thm]{Lemma}
\newtheorem{cor}[thm]{Corollary}
\theoremstyle{definition}
\newtheorem{dfn}[thm]{Definition}
\newtheorem{assump}[thm]{Assumption}
\newtheorem{remark}{Remark}
\newtheorem{example}{Example}

\usepackage{hyperref}  

\begin{document}

\title[Coupling Markov Chains with a Common Image Chain]{Coupling Markov Chains \\with a Common Image Chain}
\author{Edward T.~Crane, Alexander E.~Holroyd, \and Erin G.~Russell}
\address{University of Bristol, UK}
\email{edward.crane@bristol.ac.uk}  
\email{a.e.holroyd@bristol.ac.uk}
\email{erin.russell@bristol.ac.uk}
\date{17 April 2026}



\begin{abstract}
    Consider time-homogeneous discrete-time Markov chains $X$, $Y$, and $Z$ on countable state spaces, with specified initial distributions. Suppose for maps $f$ and $g$ that $(f(X_t))_{t \ge 0}$ and $(g(Y_t))_{t \ge 0}$ are both equal in law to $Z$. We prove that $X$ and $Y$ can be coupled so that $(X_t, Y_t)_{t \ge 0}$ is a homogeneous Markov chain with $f(X_t) = g(Y_t)$ for all $t \ge 0$. Without the assumption that $Z$ is Markov, no such Markov coupling exists in general, even an inhomogeneous one. Moreover, we give an explicit construction of such a coupling, with the additional property that $X$ and $Y$ are conditionally independent given the entire trajectory $(f(X_t))_{t \ge 0}$.
    
    Under the further assumption that $X$ and $Y$ are stationary, we construct a coupling having the above properties that is also stationary. In this case, conditional independence holds for the corresponding two-sided chains indexed by $\mathbb{Z}$ (but not necessarily for the one-sided versions). 

    We prove further properties of our couplings in special cases where $f$ or $g$ satisfies the strong lumping condition (also known as Dynkin's condition) or the exact lumping condition (also known as the Pitman-Rogers condition). When $f$ is a strong lumping and $g$ is an exact lumping, we show that our coupling coincides with an intertwining of Markov chains as constructed by Diaconis and Fill. 
\end{abstract}

\maketitle

%
%

\section{Introduction}
\label{sec:1}

\begin{thm}[Homogeneous coupling]
\label{thm:DTHWMC} 
    Consider time-homogeneous Markov chains  $\Xoriginal$, $\Yoriginal$, and $\Zoriginal$ with time indexed by  $t=0,1,2,\dots$, and countable state spaces $A$, $B$, and $C$ respectively. 
    Suppose for maps $f: A \to C$ and $g: B \to C$ that
    \[\bigl(f(\Xoriginal_t)\bigr)_{t \ge 0} \eqdist \Zoriginal \eqdist \bigl(g(\Yoriginal_t)\bigr)_{t \ge 0}\,.\]
    Then there exists a homogeneous Markov chain $(\Xcopy_t, \Ycopy_t)_{t\geq0}$ such that 
     \[(\Xcopy_t)_{t \ge 0} \eqdist \Xoriginal \;\text{ and } \; (\Ycopy_t)_{t \ge 0} \eqdist \Yoriginal\,,\]
     and
     $$f(X_t)=g(Y_t)\text{ for all }t\geq 0.$$
     Writing $\Zcopy_t =f(\Xcopy_t) = g(\Ycopy_t)$ (so that $\Zcopy \eqdist \Zoriginal$), the marginal processes
     $X$ and $Y$ may be taken to be conditionally independent given the process $Z$.
\end{thm}

\begin{remark}\label{not-just-strong}
     There is no assumption in Theorem~\ref{thm:DTHWMC} that  $(f(\Xoriginal_t))_{t \ge 0}$ and $(g(\Yoriginal_t))_{t \ge 0}$ should remain Markov if $\Xoriginal$ and $\Yoriginal$ are started from arbitrary initial distributions (but with the same transition matrices).  If such an assumption is added then a much simpler proof is available, and the result is not new -- it follows from a more general result of Edalat~\cite{edalat1999}.  Maps $f$ and $g$ with this stronger property are called \emph{strong lumpings}.  The more general setting  of Theorem~\ref{thm:DTHWMC} is considerably more subtle. Similar remarks apply to Theorem~\ref{thm:stationary} and Corollaries~\ref{thm:DTIWMC} and \ref{cor:quasistationary} below.  We discuss these matters in detail in
     Sections~\ref{S:weak}--\ref{ss:EMCs}; also see Remark~\ref{remark-strong-exact}. 
     In particular, we establish further consequences under the assumption that one or both of $f$ and $g$ are strong lumpings, as well as under the related \emph{exact lumping} condition.   
\end{remark}

\begin{remark}
    If the assumption in Theorem~\ref{thm:DTHWMC} that the image process $\Zoriginal$ is Markov is omitted then the conclusion may fail.  
    An example showing this can be obtained from a pair of distinct finite trees whose random walk return times to the root have the same distribution, as given by Benjamini, Kozma, Lov\'asz, Romik and Tardos ~\cite{benjamini2006waiting}. See Example~\ref{ex:nonMarkovZnocoupling} in Section~\ref{sec:examples} for details. 
\end{remark}

\begin{remark}
    The conclusion that the coupling $(\Xcopy_t, \Ycopy_t)_{t\geq0}$ is Markov is a key point of Theorem~\ref{thm:DTHWMC}. Indeed, straightforward  abstract considerations show that under the assumptions of the theorem there exists a stochastic process $(\Xcopy_t, \Ycopy_t)_{t\geq0}$ satisfying all the claimed properties besides the Markov property.  See also Remark~\ref{cond-remark} below.
\end{remark}
\begin{remark}\label{not-just-stepwise}
    Conditioning on the entire trajectory $Z = (Z_t)_{t \ge 0}$ is in general not the same as conditioning only one step ahead, i.e.\ letting $X_{t+1}$ and $Y_{t+1}$ be conditionally independent given $X_t$, $Y_t$, and $Z_{t+1}$. (See Example~\ref{ex:biasedRW} in Section~\ref{sec:examples}.) However, these are the same when $f$ and $g$ are strong lumpings -- see also Remark~\ref{remark-strong-exact} below. Making $X$ and $Y$ conditionally independent given $Z$ is also not the same as conditioning independent processes distributed like $\Xoriginal$ and $\Yoriginal$ on the (possibly null) event that $f(X_t) = g(Y_t)$ for all $t \ge 0$. 
\end{remark}

Theorem~\ref{thm:DTHWMC} can be adapted to time-inhomogeneous Markov chains via the straightforward device of including the time parameter in the state, leading to the following.
\begin{cor}[Inhomogeneous coupling]
\label{thm:DTIWMC}
     If the time-homogeneity assumption on $\Xoriginal$, $\Yoriginal$ and $\Zoriginal$ in Theorem~\ref{thm:DTHWMC} is dropped, then the conclusions still hold, except with the resulting coupling $(X,Y)$ no longer necessarily homogeneous. 
\end{cor}

We have the following variant for stationary Markov chains.  In this case time must be two-sided in order for the conditional independence statement to hold.  A stochastic process $(X_t)_{t\in\mathbb{Z}}$ is \textbf{stationary} if it is equal in distribution to $(X_{t+1})_{t\in\mathbb{Z}}$. 

\begin{thm}[Stationary coupling]
\label{thm:stationary}
    Consider stationary time-homogeneous Markov chains  $\Xoriginal$, $\Yoriginal$, and $\Zoriginal$ with time indexed by  $t\in\mathbb{Z}$.
    Suppose for maps $f: A \to C$ and $g: B \to C$ that \[(f(\Xoriginal_t))_{t \in \mathbb{Z}} \eqdist \Zoriginal \eqdist (g(\Yoriginal_t))_{t \in \mathbb{Z}}.\]
    Then there is a stationary Markov chain $(X_t,Y_t)_{t \in \mathbb{Z}}$ such that $(X_t)_{t \in \mathbb{Z}} \eqdist \Xoriginal$, $(Y_t)_{t \in \mathbb{Z}} \eqdist \Yoriginal$, and $f(X_t) = g(Y_t)$ for all $t \in \mathbb{Z}$. Moreover, $(X_t)_{t \in \mathbb{Z}}$ and $(Y_t)_{t \in \mathbb{Z}}$ may be taken to be conditionally independent given $(f(X_t))_{t \in \mathbb{Z}}$. 
\end{thm}
The next corollary concerns the following variant. A homogeneous Markov chain $X=(X_t)_{t\ge 0}$ is called \textbf{quasistationary} if it has an absorbing state $\rho$ and an absorption probability $\lambda\in(0,1)$ such that $\mathbb{P}(X_t \neq \rho)  = (1-\lambda)^t$ for all $t \ge 0$, and $\mathbb{P}(X_t = a \mid X_t \neq \rho)$ does not depend on $t$. 
\begin{samepage}
\begin{cor}[Quasistationary coupling]
\label{cor:quasistationary}
    Let $\Xoriginal$, $\Yoriginal$ and $\Zoriginal$ be quasistationary Markov chains, with state spaces $A$, $B$, and $C$, and absorbing states $\rho_A$, $\rho_B$, and $\rho_C$ respectively. Suppose for maps $f:A \to C$ and $g: B \to C$ satisfying $f^{-1}(\rho_C) = \{\rho_A\}$ and $g^{-1}(\rho_C) = \{\rho_B\}$ that 
    \[
    (f(\Xoriginal_t))_{t \ge 0} \eqdist \Zoriginal \eqdist (g(\Yoriginal_t))_{t \ge 0}.
    \]
    Then  $\Xoriginal, \Yoriginal$ and $\Zoriginal$ have equal absorption probabilities, say $\lambda$, and there is a quasistationary Markov chain $(X_t,Y_t)_{t \ge 0}$ with 
    absorbing state $(\rho_A, \rho_B)$ and absorption probability $\lambda$ such that $(X_t)_{t \ge 0} \eqdist \Xoriginal$, $(Y_t)_{t \ge 0} \eqdist \Yoriginal$, and $f(X_t) = g(Y_t)$ for all $t \ge 0$. In particular, 
    the absorption times of the coordinates coincide. The marginal chains $X$ and $Y$ may be taken to be conditionally independent given $(f(X_t))_{t \ge 0}$.  
\end{cor}
\end{samepage}

The above corollary may be applied to any two quasistationary chains $\Xoriginal$ and $\Yoriginal$ that have equal absorption probabilities, by taking $Z$ to be a two-state quasistationary chain with the same absorption probability.

We have the following natural extension of Theorem~\ref{thm:DTHWMC} to three or more chains, which is easily proved by sequential applications of Theorem~\ref{thm:DTHWMC}. 
\begin{cor}
\label{cor:DTHWMC (of many processes)} 
    Consider homogeneous Markov chains  $\Xoriginal^{(1)},\ldots,\Xoriginal^{(k)}$ with time indexed by $t\geq 0$, and countable state spaces $A_1,\dots,A_k$, respectively.
    For a homogeneous Markov chain $\Zoriginal$ with state space $C$ and maps $f_i: A_i \to C$, suppose that
    \[
        \bigl(f_i(\Xoriginal^{(i)}_t)\bigr)_{t \ge 0} \eqdist \Zoriginal\quad\text{for all } i=1,\dots,k\,.
    \]
    Then there exists a homogeneous Markov chain $(\Xcopy^{(1)}_t,\dots,\Xcopy^{(k)}_t)_{t\geq0}$ such that
   $ 
        (\Xcopy^{(i)}_t)_{t \ge 0} \eqdist \Xoriginal^{(i)}
    $ for all $i$,
    and
    $
        f_i(X^{(i)}_t)=f_j(X^{(j)}_t)
    $ for all $i,j$ and $t\geq 0$. 
    Writing $\Zcopy_t=f_1(\Xcopy^{(1)}_t)=\dots =f_k(\Xcopy^{(k)}_t)$ (so that $\Zcopy \eqdist \Zoriginal$), the marginal processes
    $X^{(1)},\dots,X^{(k)}$ may be taken to be conditionally independent given~$Z$.
\end{cor}

In this paper, a \textbf{Markov chain} always has countable state space, and time is usually indexed by $\mathbb{N} = \{0,1,2,\dots\}$, or by $\mathbb{Z}$ in the stationary case. Occasionally we will refer to continuous-time pure-jump Markov chains with time indexed by $[0,\infty)$, but they are not the focus of this paper.  We regard a Markov chain as a stochastic process, and hence to specify its distribution requires knowledge of the distribution of its state at time $0$ as well as its transition probabilities. 
We write $\MC(\upalpha,P)$ for the discrete-time homogeneous Markov chain (\textbf{DTHMC}) with time indexed by $\mathbb{N}$ that is specified by initial distribution $\upalpha$ and transition matrix $P$.

\subsection*{The construction}

The initial distribution $\upalpha$ and transition matrix $P$ of a coupling $\MC(\upalpha,P)$ satisfying the conclusions of Theorem~\ref{thm:DTHWMC} can be described explicitly, as we explain now.  Assume without loss of generality that each state in $A$ has positive probability of being visited by $\Xoriginal$, and similarly for $B$ and $\Yoriginal$.  Let
\begin{equation}
\label{eqn:delta}
    \Delta = \Delta(f,g)=\bigl\{(a,b) \in A \times B:\, f(a) = g(b)\bigr\},
\end{equation}
and define a matrix $R$ with rows and columns indexed by $\Delta$ as follows.
\begin{equation}\label{eq: R def}
    R\bigl((a,b),(a',b')\bigr) = 
    \begin{cases}
        \frac{\displaystyle{P_X(a,a')P_Y(b,b')}}{\displaystyle{P_Z(f(a),f(a'))}} & \text{if $P_Z(f(a),f(a')) > 0$,}\\
        0 & \text{otherwise.}
    \end{cases}
\end{equation}

Let $\mathbf{1}_\Delta$ denote the column vector of $1$s indexed by $\Delta$. We will show that the coordinatewise limit
\[
    \varphi = \lim_{m \to \infty} R^m \mathbf{1}_{\Delta}
\]
exists and is a non-negative right eigenvector of $R$ with eigenvalue $1$. Let
\[
    \Delta' = \bigl\{ (a,b) \in \Delta \,:\, \varphi(a,b) > 0\bigr\}\,.
\]

For $(a,b) \in \Delta'$, let $c = f(a) = g(b)$ and 
\begin{equation}\label{eq: coupled pi def}
    \upalpha(a,b) = 
    \begin{cases} 
        \frac{\displaystyle{\upalpha_X(a)\upalpha_Y(b)}}{\displaystyle{\upalpha_Z(c)}} \,\varphi(a,b) & \text{if $\upalpha_Z(c) > 0$}, \\
        0 & \text{if $\upalpha_Z(c) = 0.$}
    \end{cases}
\end{equation}
For $(a,b), (a',b') \in \Delta'$, let
\begin{equation}
\label{eq: coupled P def}
    P((a,b),(a',b')) = \frac{1}{\varphi(a,b)} R\bigl((a,b),(a',b')\bigr)\,\varphi(a',b')\,.
\end{equation}

Then $\upalpha$ is a probability vector and $P$ is a stochastic matrix, indexed by~$\Delta'$. Moreover, $\MC(\upalpha,P) = (\Xcopy_t, \Ycopy_t)_{t\geq0}$ satisfies the claims of Theorem~\ref{thm:DTHWMC}, including the conditional independence statement. Notice that equation~\eqref{eq: coupled P def} looks formally like a Doob $h$-transform, but it is not one in general because $R$ is not necessarily stochastic.
The above claims about $R$, $\upalpha$, and $P$ are proved in Section~\ref{sec:2}, as part of the proof of Theorem~\ref{thm:DTHWMC} itself.

Example~\ref{ex:stationarymarginals} in Section~\ref{sec:examples} shows that when $\Xoriginal$ and $\Yoriginal$ are stationary, the coupled chain $(\Xcopy_t,\Ycopy_t)_{t \ge 0}$ constructed in our proof of  Theorem~\ref{thm:DTHWMC} is not necessarily stationary.

Our proof of Theorem~\ref{thm:stationary} is similar to that of Theorem~\ref{thm:DTHWMC}. We construct a coupling by specifying an explicit transition matrix and stationary distribution, and prove that it has the claimed conditional independence property. To explain the construction, let $\Xoriginal$, $\Yoriginal$, and $\Zoriginal$ have (stationary) time-$0$ distributions $\uppi_X, \uppi_Y$, and $\uppi_Z$, and transition matrices $P_X$, $P_Y$, and $P_Z$. We may assume without loss of generality that the supports of $\uppi_X$, $\uppi_Y$, and $\uppi_Z$ are all of $A$, $B$, and $C$, respectively, since any state not in these supports is never visited. Let the time-reversals of $\Xoriginal$, $\Yoriginal$, and $\Zoriginal$ have transition matrices  $P_X^{\rev}$, $P_Y^\rev$, and $P_Z^\rev$. Let $\Delta$, $R$, and $\varphi$ be defined as they were in the setting of Theorem~\ref{thm:DTHWMC}, and define $R^\rev$ and $\varphi^\rev$ analogously using $P_X^\rev$, $P_Y^\rev$, and $P_Z^\rev$. We will show that the conditionally independent coupling described in Theorem~\ref{thm:stationary} takes its values in the set 
\[
    \Delta'' = \{(a,b) \in \Delta \,:\, \varphi(a,b) > 0 \text{ and } \varphi^\rev(a,b) > 0\}.
\]
It has transition matrix $P$ given by equation~\eqref{eq: coupled P def}, i.e.~the same formula as for the non-stationary case, but restricted to $\Delta''$. Its stationary distribution~$\uppi$ (the distribution of $(X_0,Y_0)$) is given by
\begin{equation}
\label{eq: stationary dist}
    \uppi(a,b) = 
    \begin{cases}        
        \displaystyle{
        \frac{\uppi_X(a)\uppi_Y(b)}{\uppi_Z(c)}
        }
        \varphi(a,b)\varphi^\rev(a,b) 
        & \text{if $\uppi_Z(c) > 0$}, \\
        0 
        & \text{if $\uppi_Z(c) = 0.$}
    \end{cases}
\end{equation}

\begin{remark}
    Theorem~\ref{thm:stationary} may be stated in the language of ergodic theory. Let $\Xoriginal$ and $\Yoriginal$ be two-sided Markov shifts over a countable alphabet, having a common factor $\Zoriginal$ that is also a two-sided Markov shift, where the factor maps are  one-block maps. Theorem~\ref{thm:stationary} states that the \emph{relatively independent joining} of $\Xoriginal$ and $\Yoriginal$ with respect to $\Zoriginal$ is also a two-sided Markov shift for which the projection maps to $\Xoriginal$ and $\Yoriginal$ are one-block maps. For further information about joinings, see~\cite{delaRue2006}. 
\end{remark}

\begin{remark}\label{cond-remark}
    To prove Theorems~\ref{thm:DTHWMC} and~\ref{thm:stationary} we will not take the conditional independence property as the \emph{definition} of the coupled chain. This would be possible using the disintegration theorem, since the processes $\Xoriginal$, $\Yoriginal$, and $\Zoriginal$ are random variables taking their values in Polish spaces. Indeed, the Gluing Lemma~\cite[Lemma 7.6]{villani2021topics} ensures that the conditionally independent coupling exists and is unique in law. However, the resulting abstract representation of the coupling 
    does not seem to lend itself readily to proving the Markov property.  (An exception is the special case when the entire process $\Zoriginal$ has countable support, i.e.\ where absorption occurs almost surely.  In this case elementary conditioning suffices, and a direct proof appears in  the PhD thesis of the third author~\cite{russell2025thesis}.) Instead, we construct the coupled chain by giving a transition matrix and initial condition, as explained above, and check that the marginals have the correct law and are conditionally independent given their common image. 
    Without the conditional independence condition, there may exist other couplings satisfying the remaining conclusions of Theorem~\ref{thm:DTHWMC}; for instance see Example~\ref{ex:biasedRW} in Section~\ref{sec:examples}.
\end{remark}
\begin{remark} \textbf{Strong and exact lumpings.} \label{remark-strong-exact}
    As stated in Remark~\ref{not-just-strong}, under the additional assumption that the maps $f$ and $g$ in Theorem~\ref{thm:DTHWMC} are strong lumpings, a much simpler proof is available.  More precisely, in this case we can simply couple $X$ and $Y$ to be conditionally independent given $Z$ at each time step, as described in Remark~\ref{not-just-stepwise}.  In the resulting coupling, the projection maps from $(X,Y)$ to $X$ and $Y$ are also strong lumpings.  More interestingly, we show in Proposition~\ref{lem: g strong proj1 strong} that if just one map, $f$, is assumed to be a strong lumping then the projection from $(X,Y)$ to $X$ is a strong lumping.  We prove a similar fact, Proposition~\ref{cor: g exact proj1 exact}, for the related notion of exact lumping.  In each case the coupling is the same as the one in Theorem~\ref{thm:DTHWMC}, but the proof is much simpler.
    
    Further, we show in Section~\ref{ss:intertwining} that if $f$ is a strong lumping and $g$ is an exact lumping then the coupling of Theorem~\ref{thm:DTHWMC} is an \emph{intertwining} of the Markov chains $\Xoriginal$ and $\Yoriginal$, as introduced by Diaconis and Fill in their seminal paper \cite{diaconis1990duality} about strong stationary times.
\end{remark}

\subsection*{Organisation of the paper}

In Sections~\ref{S:weak}--\ref{ss:intertwining} we explain some key concepts of lumping and coupling of Markov chains and describe how our main theorems relate to existing results about couplings of Markov chains. These sections may be read independently of Sections~\ref{sec:2} and \ref{sec:others}, which give the proofs of our main theorems.

Section~\ref{S:weak} describes the concepts of weak lumping and weak Markovian couplings, briefly discusses the special cases of strong lumping and exact lumping, and states Propositions~\ref{lem: g strong proj1 strong} and \ref{cor: g exact proj1 exact}, the refinements of Theorem~\ref{thm:DTHWMC} which apply when $f$ or $g$ is either a strong or exact lumping, as described in Remark~\ref{remark-strong-exact} above.
In Section~\ref{ss:SMC} we give further details about strong lumping, define strong Markovian couplings, discuss their previous appearances in the literature, and prove Proposition~\ref{lem: g strong proj1 strong}. 
In Section~\ref{ss:EMCs} we give further details about exact lumping, define exact Markovian couplings, and prove Proposition~\ref{cor: g exact proj1 exact}.
In Section~\ref{ss:intertwining} we address intertwinings and the case where $f$ is a strong lumping and $g$ is an exact lumping, as discussed in Remark~\ref{remark-strong-exact}.

In Section~\ref{sec:2}, we prove the main result, Theorem~\ref{thm:DTHWMC}.  We give a brief outline in Section~\ref{ss:outline}, followed by the detailed steps in Sections~\ref{ss:defW_m}--\ref{ss:condindep}. 
Section~\ref{sec:2} does not depend on anything from Sections~\ref{S:weak}--\ref{ss:intertwining}.  

Section~\ref{sec:others} contains the proofs of the additional results, Theorem~\ref{thm:stationary} and Corollaries~\ref{thm:DTIWMC},~\ref{cor:quasistationary},~and~\ref{cor:DTHWMC (of many processes)}.  The proof of Theorem~\ref{thm:stationary} follows a similar pattern to the proof of Theorem~\ref{thm:DTHWMC}, making adjustments to account for the two-sidedness and stationarity of the chains.  Corollary~\ref{cor:quasistationary} is deduced from~Theorem~\ref{thm:stationary}, while Corollaries~\ref{thm:DTIWMC} and \ref{cor:DTHWMC (of many processes)} follow from Theorem~\ref{thm:DTHWMC}.  
Section~\ref{sec:examples} describes several examples illustrating our main results.

\subsection*{Forthcoming work.}

In a forthcoming companion paper~\cite{crane2026part2} we analyse the situation of Theorem~\ref{thm:DTHWMC} in more detail, using the algebraic characterization of weak lumping of homogeneous Markov chains that was given by Gurvits and Ledoux~\cite{gurvits2005markov}. This enables us to identify an explicit set of initial distributions $\upalpha'$ for which $\MC(\upalpha',P)$ is a coupling of $\Xoriginal$ and $\Yoriginal$, that in some cases contains distributions other than $\upalpha$. We also prove some further results there, for example an extension to uniformizable chains in continuous time and some partial converse results which characterize when given homogeneous Markov chains may be coupled by a strong Markovian coupling, or by an exact Markovian coupling whose states are restricted to a given block-diagonal subset of the product of their state spaces.

In another forthcoming paper~\cite{crane2026algorithm} we study the general problem of coupling any two homogeneous Markov chains so that the coupled process is a homogeneous Markov chain taking its values in a given subset of the possible states and making its transitions in a given subset of the possible transitions. The main result there is that this problem is decidable when the state spaces are finite. One of the examples in~\cite{crane2026algorithm} almost fits the setup of Theorem~\ref{thm:DTHWMC}, except that the common image process is not Markov; in that example the conditionally independent coupling is not Markov, but nevertheless there does exist another coupling taking values in $\Delta(f,g)$ that is Markov.

\section{Weak lumping and weak Markovian couplings}\label{S:weak}
We now explain some key concepts which will help us to describe how our main theorems relate to known results about couplings of Markov chains. 
For any Markov chain $X$, and a function $f$ on its state space, we use the shorthand $f(X)$ for the process $(f(X_t))_{t\geq0}$. For a countable set $U$, let $\PM(U)$ denote the set of probability distributions on $U$, and let $\SM(U)$ denote the set of stochastic matrices with rows and columns indexed by $U$. 

Recall that for a measurable map $f: A \to B$ and a probability measure $\upalpha$ on $A$, the \emph{pushforward} of $\upalpha$ under $f$ is the probability measure $f_*(\upalpha)$ on $B$ defined by $f_*(\upalpha)(S) = \upalpha(f^{-1}(S))$ for all measurable sets $S \subseteq B$. In this paper we work with countable state spaces, where all subsets are measurable, so we shall not need to mention measurability again.

\begin{dfn}[Weak lumping]
\label{def:WL}
    We say that a stochastic matrix $P\in\SM(A)$ is \textbf{weakly lumpable} with respect to a surjective map $f: A \twoheadrightarrow B$ if there exists $\upalpha\in\PM(A)$ such that $f(\MC(\upalpha,P))$ is a homogeneous Markov chain taking values in $B$. In this case, we say that $\MC(\upalpha,P)$ \textbf{lumps weakly} under $f$, and that $f$ is a \textbf{weak lumping} of $\MC(\upalpha,P)$. If the transition matrix of $f(\MC(\upalpha,P))$ is $Q$ then we say that $P$ is weakly lumpable to $Q$ under $f$, and we say that $\MC(\upalpha,P)$ lumps weakly to $\MC(f_*(\upalpha),Q)$ under $f$. 
\end{dfn}

\begin{remark}\label{rem:Ore}
 There is a category whose objects are DTHMCs with countable state spaces and whose morphisms are weak lumpings. (The composition of two weak lumpings is a weak lumping.) Theorem~\ref{thm:DTHWMC} says that this category has the \emph{Ore property}. This means that for any three objects $X$, $Y$, and $Z$ in the category with morphisms $f: X \to Z$ and $g: Y \to Z$, there exists a \emph{semi-pullback}, i.e. an object $W$ and morphisms $h:W \to X$ and $i:W \to Y$ such that $f \circ h = g \circ i$.  (A semi-pullback need not satisfy the universal property that must be satisfied by a pullback.)
\end{remark}

 An elegant algebraic characterization of weak lumpability of DTHMCs with finite state spaces was given by Gurvits and Ledoux \cite{gurvits2005markov} in 2005, following a more complicated characterization given by Rubino and Sericola~\cite{rubino1991finite} in 1991.  
 There are many other terms in the literature associated with lumping of Markov chains, for example \emph{aggregation}, \emph{projection}, \emph{collapsing}, \emph{Markov functions}, and \emph{bisimulation}. Another relevant term is \emph{lifting} of Markov chains, which refers to the situation where a chain $Y$ is given and a chain $X$ and a map $f$ are constructed such that $f(X) \eqdist Y$ and $X$ is easier to understand than $Y$. This diversity of terminology reflects the fact that Markov chain lumping is a very useful idea that has been discovered and used in many different areas of pure and applied probability and statistics. General processes of the form $f(X)$ where $X$ is Markov, and $f$ is a deterministic or randomizing function, are known as \emph{hidden Markov chains} and they are used in many applications.

\begin{dfn}[Weak Markovian coupling]
\label{def:WMC}
    Suppose $\Xoriginal=\MC(\upalpha_X,P_X)$ and $\Yoriginal=\MC(\upalpha_Y,P_Y)$ are DTHMCs on state spaces $A$ and $B$. A \textbf{weak Markovian coupling} (WMC) of $\Xoriginal$ and $\Yoriginal$ is a DTHMC $(\Xcopy,\Ycopy)=\MC(\mu,P)$ on $A \times B$ whose marginal processes satisfy $\Xcopy \eqdist \Xoriginal$ and $\Ycopy \eqdist \Yoriginal$.   
\end{dfn}

Any WMC is a Markov chain that lumps weakly under the projections
\[
    \proj^1: A \times B \to A \;\text{ and }\; \proj^2: A \times B \to B.
\] 
The coupling described in Theorem~\ref{thm:DTHWMC} is a WMC.

\begin{remark}
 The analogue of WMCs in the context of continuous-time homogeneous Markov chains with finite state spaces are called \emph{weak Markov copulae} by Bielecki, Jakubowski and Niew\k{e}g\l{}owski \cite{bielecki2013intricacies} and more recently \emph{weak Markov chain structures} in their book \cite[\S7.2]{BJNbook}. In both works, the property that the marginal processes of the coupled Markov chain are Markov processes in the filtrations that they generate (i.e. the fact that the projection maps are weak lumpings) is termed \emph{weak Markov consistency}. The book \cite{BJNbook} sets up technical machinery that can be used for various applications of continuous-time Markov chain couplings for example to model dependence between multiple financial instruments or dependence between time series in seismology. See also the PhD thesis of Chang \cite{Chang2017thesis} for further examples and applications.
\end{remark}

In Sections~\ref{ss:SMC},~\ref{ss:EMCs}, and ~\ref{ss:intertwining} below, we discuss three important special cases of the situation of Theorem~\ref{thm:DTHWMC}, involving two special types of weak lumping which we now describe.

If $X$ is a Markov chain and for each time $t$ the distribution of $f(X_{t+1})$ given $X_t$ is a function only of $f(X_t)$, then $f$ is called a \textbf{strong lumping}. See Section~\ref{ss:SMC} for other characterizations of strong lumping, one of which is Dynkin's criterion. Not all weak lumpings are strong lumpings.

\begin{dfn}[Strong lumping]
\label{def:SL} 
    Let $f: A \twoheadrightarrow C$ be a surjective map between countable sets. We say that a stochastic matrix $P\in\SM(A)$ \textbf{lumps strongly} under $f$ if there exists a stochastic matrix $Q\in\SM(C)$ such that for all $\upalpha\in\PM(A)$, the chain $\MC(\upalpha,P)$ lumps weakly to $\MC(f_*(\upalpha),Q)$. In this case, we say that $P$ lumps strongly under $f$ to $Q$.
\end{dfn}

\begin{remark}
    Other names for strong lumping in the literature are \emph{ordinary lumping} (e.g. in \cite{Buchholz}) and often (potentially confusingly) just \emph{lumping} or \emph{lumpability}. In the concurrency theory literature, strong lumpings are sometimes called \emph{simulation morphisms}, and the equivalence relation whose classes are the fibres of $f$ is called a \emph{bisimulation}. These concepts apply to a generalization of Markov chains in concurrency theory called \emph{labelled Markov chains}. These model both random evolution independent of history and interaction with a non-deterministic controller and observer. 
\end{remark}

In the setting of Theorem~\ref{thm:DTHWMC}, when $f$ is a strong lumping it turns out that the algebra simplifies. We will prove the following result in Section~\ref{ss:SMC}; the proof is much simpler than the proof of Theorem~\ref{thm:DTHWMC}. 
\begin{prop}
\label{lem: g strong proj1 strong}
    In the setting of Theorem~\ref{thm:DTHWMC}, suppose $f$ is a strong lumping of $P_X$ and $B$ contains no states that are almost surely never visited by $\Yoriginal$. Then $R$ is a stochastic matrix, so $\varphi \equiv 1$, $\Delta'=\Delta$, and $P=R$. Moreover, the projection map $\proj^2: A \times B \to B$ is a strong lumping of $P$ to $P_Y$.
\end{prop}

\begin{remark}
    The special case of Proposition~\ref{lem: g strong proj1 strong} in which \emph{both} $f$ and $g$ are strong lumpings was proved by Edalat~\cite{edalat1999}, as a special case of a similar statement in which the state spaces were allowed to be non-discrete. In Edalat's main result the state spaces must be analytic spaces and the maps $f$ and $g$ must be Borel measurable strong lumpings. In other words, Edalat showed that a category of DTHMCs with strong lumpings as morphisms is a category with the Ore property, as mentioned in Remark~\ref{rem:Ore}. We warn the reader that Edalat uses the word `stationary' for what we call `time-homogeneous' when describing Markov processes. Edalat applied his result to prove the transitivity of a bisimulation-based equivalence relation between labelled Markov processes; see \cite{desharnais2002} for the context. 
\end{remark}

Couplings of Markov chains in which both projection maps to the marginal chains are strong lumpings have been considered in many other works under a variety of names, as we discuss in Section~\ref{ss:SMC}. 

Another well-known sufficient condition for $f(X_t)_{t \ge 0}$ to be a homogeneous Markov chain is the \textbf{exact lumping} condition, also known (especially in the context of continuous time) as the \emph{Pitman--Rogers intertwining condition}. Section~\ref{ss:EMCs} below gives more details. It seems that the vast majority of weak lumpings in the probability literature are either strong lumpings or exact lumpings (and the exact lumpings tend to be the `surprising' ones.) So it is of interest to know what  happens when one of the inputs to Theorem~\ref{thm:DTHWMC} is an exact lumping. We will prove the following corollary in Section~\ref{ss:EMCs}.

\begin{prop}\label{cor: g exact proj1 exact}
In the setting of Theorem~\ref{thm:DTHWMC}, suppose that $f$ is an exact lumping of $\MC(\upalpha_X,P_X)$ to $\MC(\upalpha_Z,P_Z)$, and that $B$ contains no states that are almost surely never visited by $\Yoriginal$. Let $\MC(\upalpha,P)$ be the coupling defined by equations~\eqref{eq: coupled pi def}~and~\eqref{eq: coupled P def}. Then the projection $\proj^2:A \times B \to B$ is an exact lumping of $\MC(\upalpha,P)$ to $\MC(\upalpha_Y,P_Y)$.
\end{prop}
See Remark~\ref{rem:stationaryexactlumping} below for a related statement about the stationary setting of Theorem~\ref{thm:stationary}.

\begin{remark}
     Our results above can be adapted to uniformizable countable-state Markov chains in continuous time, by uniformizing the marginal chains and embedding the discrete-time couplings into continuous time. In the interest of brevity we have not included these arguments.
\end{remark}

\section{Strong lumping and strong Markovian couplings}
\label{ss:SMC}

\begin{lem}[Dynkin's criterion in discrete time]
\label{lem:dynkin}
    The stochastic matrix $P\in\SM(A)$ lumps strongly under $f: A \twoheadrightarrow C$ if and only if for each $z,z'\in C$,
    \begin{equation} \label{eq: Dynkin condition}
        \sum_{x'\in f^{-1}(z')}{P(x,x')}
        \quad
        \text{ takes the same value for all $x\in f^{-1}(z)$.}
    \end{equation}
\end{lem} 
\begin{lem}
\label{lem: strong lumping cond indep}
    If $P \in \SM(A)$ lumps strongly under $f:A \twoheadrightarrow C$ and $X = \MC(\upalpha,P)$, where $\upalpha$ is any probability distribution on $A$,  then for each $t \ge 0$, $X_t$ and $(f(X_s))_{s > t}$ are conditionally independent given $f(X_t)$.
\end{lem}
A proof of Lemma~\ref{lem:dynkin} may be found in \cite[Thm.~4.1.18]{rubino2014markov}, and Lemma~\ref{lem: strong lumping cond indep} may be found in~\cite[Prop. 2.25]{crane2024weak}.   A converse of Lemma~\ref{lem: strong lumping cond indep} holds, that if $X_t$ and $f(X_s)_{s \ge t}$ are conditionally independent of $f(X_t)$ for every $t \ge 0$ for \emph{every} choice of $\upalpha \in \PM(A)$, then in particular this is true when $t = 0$ and $\upalpha = \delta_x$, for any $x \in A$, so Dynkin's condition holds and therefore $P$ lumps strongly under $f$. On the other hand, if we only know that this conditional independence property holds for some particular $\upalpha$, then to conclude that $P$ lumps strongly under $f$ we need an additional hypothesis, for example that $P$ is irreducible. 

\begin{dfn}[Strong Markovian coupling]
\label{def:SMC}
    For two stochastic matrices $P_X\in\SM(A)$ and $P_Y\in\SM(B)$, a \textbf{strong Markovian coupling} (SMC) of $P_X$ and $P_Y$ is a stochastic matrix $P\in\SM(A \times B)$ such that, for all $\upalpha\in\PM(A \times B)$, the marginals of $\MC(\upalpha,P)$ are DTHMCs driven by $P_X$ and $P_Y$. In other words, $P$ lumps strongly under $\proj^1$ to $P_X$ and also lumps strongly under $\proj^2$ to $P_Y$.
\end{dfn}

\begin{remark}
Every SMC combines with an arbitrary initial distribution on $A \times B$ to yield a WMC, since strong lumpings combine with arbitrary initial distributions to give weak lumpings.
\end{remark}

\begin{proof}[Proof of Proposition~\ref{lem: g strong proj1 strong}]
We prove that the matrix $R$ defined by equation~\eqref{eq: R def} is stochastic, starting from Dynkin's condition for $f$. 
Because $g(\Yoriginal) \eqdist \Zoriginal$ and  each element of $B$ is visited by $\Yoriginal$ with positive probability at some time, if $P_Z(c,c') = 0$ then for every $b \in g^{-1}(c)$ and $b' \in g^{-1}(c')$ we must have $P_Y(b,b') = 0$. (For otherwise there would be a positive probability that at some time $\Yoriginal$ makes a transition from $b$ to $b'$, in which case $g(Y)$ makes a transition from $c$ to $c'$, which is a null event under the law of $\Zoriginal$.) 
For any $(a,b) \in \Delta$ with $f(a) = g(b) = c$, we have
\begin{align*}
\begin{split}
    \sum_{(a',b') \in \Delta} &R\bigl((a,b),(a',b')\bigr)  =  \sum_{\substack{c' \in C\,:\\P_Z(c,c') > 0}}\sum_{\substack{a' \in f^{-1}(c')\\ b' \in g^{-1}(c')}} \frac{P_X(a,a')P_Y(b,b')}{P_Z(c,c')}
    \\
    & =  \sum_{\substack{c' \in C\,:\\P_Z(c,c') > 0}} \biggr(\sum_{b' \in g^{-1}(c')} P_Y(b,b')\biggr)\biggl(\frac{1}{P_Z(c,c')}\sum_{a' \in f^{-1}(c')} P_X(a,a')\biggr)
    \\
    & = \sum_{\substack{c' \in C\,:\\P_Z(c,c') > 0}} \sum_{b' \in g^{-1}(c')}P_Y(b,b') = \sum_{b' \in B'} P_Y(b,b') = 1.
\end{split} 
\end{align*}
We have shown that $R$ is stochastic. This implies that $\varphi_m \equiv 1$ for all $m$ and hence $\varphi \equiv 1$, so $\Delta' = \Delta$ and $P = R$. 

Finally, we must prove that $\proj^2: \Delta \to B$ is a strong lumping of $P$ to $P_Y$, which we do by verifying Dynkin's condition for $\proj^2$. Let $b,b' \in B$, $c = g(b)$, $c' = g(b')$, and let $a \in f^{-1}(c)$. If $P_Z(c,c') = 0$ then $P_Y(b,b') = 0$ and $P((a,b),(a',b') = 0$ by definition for every $(a',b') \in (\proj^2)^{-1}(b')$, so Dynkin's condition holds for this case. Now suppose $P_Z(c,c') > 0$. We have
\begin{align*}
\begin{split}
\sum_{(\widehat{a},\widehat{b}) \in (\proj^2)^{-1}(b')} P\bigl((a,b),&(\widehat{a},\widehat{b})\bigr)  =  \sum_{\widehat{a} \in f^{-1}(c')} \frac{P_X(a,\widehat{a})P_Y(b,b')}{P_Z(c,c')}\\
& =  P_Y(b,b')\cdot\frac{1}{P_Z(c,c')} \sum_{\widehat{a} \in f^{-1}(c)} P_X(a,\widehat{a})
 =  P_Y(b,b'),
\end{split}
\end{align*}
where the last equality follows from Dynkin's condition for $P_X$ to lump strongly under $f$ to $P_Z$.
\end{proof}

\subsection{Concepts related to strong Markovian coupling}
\label{ss: coupling background}

 SMCs are related to several existing concepts in the literature.
The concept of weak Markovian coupling is distinct from the concept of \emph{Markovian coupling} defined in \cite[p.~65]{levin2009markov}; in our language the Markovian couplings defined there are SMCs in which $A = B$ and $P_X = P_Y$. However, in some other literature, for example~\cite{angel2013avoidance,bottcher2017arxiv}, the term `Markovian coupling' has been used to denote a WMC. This inconsistency is the reason why we have chosen to avoid the term altogether.
 
Several papers discuss continuous-time analogs of SMCs under various names, with various conditions on the Markov processes being coupled. For example, Ball and Yeo \cite{ball1993lumpability} call the SMC condition \emph{marginalisability}, Chen \cite{chen1986jump} calls it \emph{marginality}, and (in the setting of time-inhomogeneous continuous-time Markov chains with finite state spaces, Bielecki et al.~\cite{bielecki2013intricacies} call an SMC a \emph{strong Markov copula}, and in their monograph~\cite{BJNbook} a \emph{strong Markov chain structure}.

 In the companion paper~\cite{crane2026part2} we discuss the concept of \emph{co-adapted} couplings that is defined in Burdzy and Kendall \cite{burdzy2000efficient} and also studied in \cite{connor2014complete,connor2008hypercube}. We show there that SMCs give rise to co-adapted couplings. Although co-adapted couplings of DTHMCs are not necessarily WMCs, we show that conversely when a WMC $(X_t,Y_t)_{t \ge 0} = \MC(\alpha,P)$ with state space $A \times B$ is a co-adapted coupling of Markov chains $\MC(\alpha_X,P_X)$ and $\MC(\alpha_Y,P_Y)$ with state spaces $A$ and $B$ respectively, and it visits each state in $A \times B$ with positive probability, then $P$ is an SMC of $P_X$ and $P_Y$. This result appears to be related to \cite[Thm. 1.17]{bielecki2013intricacies}.

L\'opez and Sanz \cite{lopez2002markovian} study couplings of non-explosive continuous-time Markov chains with generator matrices specified. They give necessary and sufficient linear programming conditions for the existence of a Markov chain coupling for which a given subset $K$ of $A \times B$ is preserved and the marginal projections $\proj^1:A \times B \to A$ and $\proj^2: A \times B \to B$ are strong lumpings (in the continuous-time sense). They call such couplings $K$-couplings. Saying that $K$ is preserved means that if $x \in K$ and $P_t(x,y) > 0$ then $y \in K$, where $P_t$ is the transition kernel of the coupled Markov chain. In our discrete-time setting, we are interested in SMCs whose transition matrix preserves the particular set $\Delta(f,g)$ defined by \eqref{eqn:delta}. 
Proposition~\ref{lem: g strong proj1 strong} implies that in the setting of Theorem~\ref{thm:DTHWMC}, if $f$ and $g$ are strong lumpings then there exists an SMC of $\Xoriginal$ and $\Yoriginal$ that preserves $\Delta(f,g)$. Example~\ref{ex: 3 states each affirmative} in Section~\ref{sec:examples} shows that in the setting of Theorem~\ref{thm:DTHWMC} a WMC may exist even when no SMC  exists that preserves $\Delta(f,g)$.

\section{Exact lumping and exact Markovian couplings}\label{ss:EMCs}

In this section we define another special type of coupling of Markov chains. It is related to another sufficient condition for weak lumping called \textbf{exact lumping}.

\begin{dfn}[Exact lumpability and exact lumping]
\label{def:EL}
    Let $f:A \twoheadrightarrow B$ be a surjection between countable sets. 
    We say that $P\in\SM(A)$ is \textbf{exactly lumpable} under $f$ with respect to $(\nu_z)_{z \in B}$ if $\nu_z \in \PM(f^{-1}(z))$ for each $z \in B$ and there exists $Q \in \SM(B)$ such that for every $z,z' \in B$ and every $x' \in f^{-1}(z')$, the \emph{exact lumping equation} holds:
    \begin{equation}
    \label{eq: EL discrete}
        \sum_{x\in f^{-1}(z)} \nu_z(x) P(x,x')  = Q(z,z') \nu_{z'}(x')\,.   
    \end{equation}
    We say $P$ is exactly lumpable under $f$ to $Q$ if such a family $(\nu_b)_{b \in B}$ exists for which \eqref{eq: EL discrete} holds. 
    
    For $\upalpha \in \PM(A)$, we say that
    $\MC(\upalpha,P)$ \textbf{lumps exactly} under $f$ to the chain $\MC(f_*(\upalpha),Q)$
    if there exists a family $(\nu_z)_{z\in B}$ and a matrix $Q$, that together satisfy equation~\eqref{eq: EL discrete}, such that $\upalpha$ is a convex combination of the measures $(\nu_z)_{z \in B}$. 
\end{dfn}

\begin{lem}
\label{lem:ELimpliesWL}
    If $\MC(\upalpha,P)$ lumps exactly under $f$ then it lumps weakly under~$f$.
\end{lem}

\begin{lem}
\label{lem: exact lumping cond indep}
    If $P \in \SM(A)$ lumps exactly under $f:A \twoheadrightarrow C$ and $X = \MC(\upalpha,P)$, where $\upalpha$ is any probability distribution on $A$, then for each $t \ge 0$, $X_t$ and $(f(X_s))_{s < t}$ are conditionally independent given $f(X_t)$.
\end{lem}

For proofs of Lemmas~\ref{lem:ELimpliesWL} and \ref{lem: exact lumping cond indep}, see~\cite[\S2.4~and~\S2.6]{crane2024weak}.

\begin{remark} 
    The earliest use that we have found of the term `exact lumping' is from 1984, in Schweitzer  \cite{schweitzer1984}. There, the state spaces were finite and each probability measure $\nu_z$ was uniform over $f^{-1}(z)$. For finite state spaces and without assuming uniformity of each $\nu_z$ over $f^{-1}(z)$, the condition~\eqref{eq: EL discrete} was given in matrix form earlier, in 1976, by Kemeny and Snell~\cite[Thm.~6.4.4]{kemeny1976finite} as a sufficient condition for weak lumping, but it was not given a name there. Equation~\eqref{eq: EL discrete} is the discrete time version of a well-known sufficient condition for weak lumping of continuous-time homogeneous Markov processes given by Pitman and Rogers \cite{Rogers1981markov} in 1981. Let $\Lambda$ be the \textbf{link matrix}, with rows indexed by $B$ and columns indexed by $A$, defined by 
    \[
        \Lambda(z,x) = \nu_z(x)\mathbf{1}(f(x) = z))\,.
    \]
    Then equation~\eqref{eq: EL discrete} may be written as 
    \[
        \Lambda P = Q \Lambda.
    \]
    This equation is called an \emph{algebraic intertwining}. The sufficient condition for weak lumping in \cite{Rogers1981markov} is the above algebraic intertwining equation, where $P$ and $Q$ are transition kernels rather than transition matrices. Another way to say that $X$ lumps exactly under $f$ to $Z$ that appears in the literature (e.g. in \cite{arnaudon2024intertwining,diaconis1990duality,swart2011intertwining}) is that $X$ is \emph{intertwined on top of} $Z$. This is a special case of a more general concept of intertwining of Markov chains, which we discuss in Section~\ref{ss:intertwining}.  
\end{remark}  

  The similarity between Lemmas~\ref{lem: strong lumping cond indep} and~\ref{lem: exact lumping cond indep} hints at a close connection between strong and exact lumping. Indeed, if the Markov chain $X$ is irreducible and is started in a stationary distribution $\uppi$ of $P$, then it has a time reversal $X^\rev=\MC(\uppi,P^\rev)$, and $P$ lumps strongly under $f$ if and only if $X^\rev$ lumps exactly under $f$. Marin and Rossi \cite{marin2014} observe the `if' direction of this statement. The `only if' part does not hold in their context as they use Schweitzer's \cite{schweitzer1984} original definition of exact lumping, where the stationary measure must be uniform over each fibre. It becomes true using our more general definition (see \cite[Cor. 10.2]{crane2024weak}).

\begin{remark}\label{rem:stationaryexactlumping}
    In the stationary setting of Theorem~\ref{thm:stationary}, if $f$ is an exact lumping of $\MC(\uppi_X,P_X)$  then $f$ is a strong lumping of $P_X^\rev$, so by Proposition~\ref{lem: g strong proj1 strong} we have $\varphi^\rev \equiv 1$. Hence the coupling provided by the proof of Theorem~\ref{thm:stationary} coincides with the one given in the proof of Theorem~\ref{thm:DTHWMC}. 
\end{remark}

\begin{dfn}[Exact Markovian coupling]
\label{def:EMC}
    Suppose $\Xoriginal$ and $\Yoriginal$ are homogeneous Markov chains on state spaces $A$ and $B$. An \textbf{exact Markovian coupling} (EMC) of $\Xoriginal$ and $\Yoriginal$ is a homogeneous Markov chain  whose state space is a subset of $A \times B$ and which lumps exactly under $\proj^1$ and $\proj^2$ to processes equal in law to $\Xoriginal$ and $\Yoriginal$ respectively. 
\end{dfn}
\begin{remark}
EMCs are WMCs, since exact lumpings are weak lumpings.
\end{remark}

Proposition~\ref{cor: g exact proj1 exact} implies that if we start in Theorem~\ref{thm:DTHWMC} with exact lumpings $f$ and $g$, then the coupling $\MC(\upalpha,P)$ that it provides is in fact an EMC. 

\begin{proof}[Proof of Proposition~\ref{cor: g exact proj1 exact}]
We take $f$ to be an exact lumping of $\MC(\upalpha_X,P_X)$ to $\MC(\upalpha_Z,P_Z)$ with respect to the family of probability vectors $(\nu_z)_{z \in C}$, where for each $z \in C$, $\nu_z$ is a probability vector on $A$ supported on $f^{-1}(z)$. These are assumed to satisfy the exact lumping equation: for every $z,z' \in C$ and every $x' \in f^{-1}(z')$,
\[\sum_{x \in f^{-1}(z)} \nu_z(x)P_X(x,x') = P_Z(z,z')\nu_{z'}(x').\]
Moreover, $\upalpha_X$ is assumed to be a convex combination of this family, and $f_*(\upalpha_X) = \upalpha_Z$ and the supports of the $\nu_c$ are disjoint, so this convex combination must be
\[ \upalpha_X = \sum_{c \in C} \upalpha_Z(c) \nu_c.\]
We construct a family $(\mu_y)_{y \in B}$ of measures on $\Delta'$, where $\mu_y$ is supported on the fibre $\left(\proj^2\right)^{-1}(y)$, by
\[\mu_y(a,b)  = \begin{cases} 0 & \text{if $b \neq y$,}\\ \nu_{g(b)}(a) \varphi(a,b) & \text{if $b = y$.}\end{cases}\]
Note that $\upalpha$ is a convex combination of the measures $\mu_y$ over $y \in B$, since for $(a,b) \in \Delta'$ with $f(a) = g(b) = c$ we have
\[\upalpha(a,b)  = \frac{\upalpha_X(a)\upalpha_Y(b)}{\upalpha_Z(c)}\varphi(a,b) = \nu_c(a)\upalpha_Y(b)\varphi(a,b) \]
so
\[\upalpha = \sum_{y \in B} \upalpha_Y(y) \mu_y.\]
 It remains to show that this family satisfies the exact lumping equation given in Definition~\ref{def:EL} for the exact lumping of $\MC(\upalpha,P)$ to $\MC(\upalpha_Y,P_Y)$ under $\proj^2: \Delta' \to B$. Let $b,b'\in B$ and $a'\in f^{-1}(g(b'))$. Let $c=g(b)$ and $c'=g(b')$. We must show that
    \begin{equation}
    \label{eqn: EL for lemma}
        \sum_{a\in f^{-1}(g(b))} 
        \nu_{g(b)}(a)\varphi(a,b) 
        P((a,b),(a',b'))  
        =
        P_Y(b,b')
        \nu_{g(b')}(a')\varphi(a',b') 
        \,.
    \end{equation}
    There are two cases. If $P_Z(c,c') = 0$ then $P_Y(b,b') = 0$, as in the proof of Proposition~\ref{lem: g strong proj1 strong}, so both sides of~\ref{eqn: EL for lemma} are $0$.  If $P_Z(c,c') > 0$ then using the definition of $P$ we see that \eqref{eqn: EL for lemma} is implied by
    \[
        \sum_{a\in f^{-1}(g(b))} 
        \nu_{g(b)}(a)
        P_X(a,a')
        =
        P_Z(c,c')
        \nu_{g(b')}(a')
        \,.
    \]
    This holds, by our assumption that $P_X$ lumps exactly under $f$ to $P_Z$ with respect to $(\nu_c)_{c \in C}$.
\end{proof}

\section{Intertwining of Markov chains}\label{ss:intertwining}
In this section we discuss another known sufficient condition for the existence of a WMC of two given DTHMCs, given by Diaconis and Fill~\cite[Thm. 2.17]{diaconis1990duality}.  Let $A$ and $B$ be countable sets\footnote{In \cite{diaconis1990duality} $B$ was taken to be a subset of the power set of $A$, because the authors were studying duality, but this is not necessary for what follows; a version of this construction for finite sets $A$ and $B$ may be found in lecture notes of Jan Swart~\cite[Prop.~3.3]{swart2013duality}.}, let $P_X \in \SM(A)$ and $P_Y \in \SM(B)$ and let $\Lambda$ be a stochastic kernel from $B$ to $A$, meaning that $\Lambda$ has non-negative entries, rows indexed by $B$ and columns indexed by $A$, and its row sums are all equal to $1$.
For $\upalpha_X \in \PM(A)$ and $\upalpha_Y \in \PM(B)$, Diaconis and Fill say that $(\upalpha_X,P_X)$ is \emph{algebraically dual} to $(\upalpha_Y,P_Y)$ with respect to the link matrix $\Lambda$ when
\[ \Lambda P_X = P_Y \Lambda \;\text{ and }\; \upalpha_X = \upalpha_Y \Lambda.\]
Under this condition, let
\[ E = \{(a,b) \in A \times B\,:\,\Lambda(b,a) > 0  \}.\]
For $(a,b) \in E$, let
\begin{equation}
\label{eqn:DiaconisFillDist}
\upalpha(a,b) = \upalpha_Y(b)\Lambda(b,a) ,\end{equation}
and for $(a,b), (a',b') \in E$, let
\begin{equation}
\label{eqn:DiaconisFillMatrix}
P((a,b),(a',b')) = \begin{cases}\displaystyle{\frac{P_Y(b,b')\Lambda(b',a')P_X(a,a')}{(P_Y \Lambda)(b,a')}}, & \text{if $(P_Y\Lambda)(b,a') > 0$},\\
0, & \text{if $(P_Y\Lambda)(b,a') = 0 $}.
\end{cases}
\end{equation}
Then $\upalpha \in \PM(E)$, $P \in \SM(E)$, and the DTHMC $(X,Y) = \MC(\upalpha,P)$ is an \emph{intertwining of the Markov chains} $\MC(\upalpha_X,P_X)$ and $\MC(\upalpha_Y,P_Y)$. This means that it is a WMC of $\MC(\upalpha_X,P_X)$ and $\MC(\upalpha_Y,P_Y)$, for which $\proj^1$ is a strong lumping and $\proj^2$ is an exact lumping with respect to the family $(\nu_b)_{b \in B}$ where $\nu_b((a,b)) = \Lambda(b,a)$ for all $a \in A$, $b \in B$.  That is,
\[ 
    \mathbb{P}(X_k = x \mid Y_0, \dots, Y_k) = \Lambda(Y_k,x)\,.
\]
Let us now consider the setting of our Theorem~\ref{thm:DTHWMC}, in the special case where $g:B \to C$ is a strong lumping and $f:A \to C$ is an exact lumping with respect to the family $(\nu_c)_{c \in C}$, then we may define a link matrix $\Lambda$ from $B$ to $A$ by  
\[
    \Lambda(b,a) = \begin{cases} \nu_{g(b)}(a) & \text{if $f(a) = g(b)$,}\\
    0, &\text{otherwise.}\end{cases}
\]
The reader may check that $(\upalpha_Y,P_Y)$ is algebraically dual to $(\upalpha_X,P_X)$ in the above sense of Diaconis and Fill. Moreover, $E \subseteq \Delta(f,g)$ and the initial distribution $\upalpha$ and transition matrix $P$ on $\Delta(f,g)$ defined by equations~\eqref{eqn:DiaconisFillDist} and~\eqref{eqn:DiaconisFillMatrix} coincide on $E$ with those constructed in our proof of Theorem~\ref{thm:DTHWMC}, since $\varphi \equiv 1$ in this case, by Proposition~\ref{lem: g strong proj1 strong}.

%
%

\section{Proof of main result}
\label{sec:2}

\subsection{Outline of the proof of Theorem~\ref{thm:DTHWMC}}
\label{ss:outline}

In Sections~\ref{ss:assumptions}, \ref{ss:jointdistofW}, and \ref{subsec:conditioning}, we show that $\varphi$ exists. Then $\upalpha$ and $P$ are well-defined. In Section~\ref{ss: pi is pm}, we show that $\upalpha$ is a probability measure, and in Section~\ref{ss: P is stochastic}, we show that $P$ is a stochastic matrix. This makes $W=\MC(\upalpha,P)$ a well-defined Markov chain. Once all objects within the statement of Theorem~\ref{thm:DTHWMC} have been constructed rigorously, the remaining proof in Section~\ref{ss: completing the proof} is relatively straightforward.

The approach of the proof is to construct a sequence of Markov chains $W^m=(X^m,Y^m)$ for $m \in \mathbb{N}$. Each $W^m$ is a coupling of $\Xoriginal$ and $\Yoriginal$ such that $W^m_t \in \Delta$ for all $t \le m$, and after that its marginals evolve independently. We compute an explicit formula for $\mathbb{P}(W^m_0 = (a,b))$, involving
\begin{equation}
\label{def: phi_m}
    \varphi_m = R^m \mathbf{1}_\Delta
    \,.
\end{equation}
Using the fact that $\mathbb{P}(W^m_0 = (a,b))$ is a bounded martingale with respect to $m$, we deduce that $\varphi_m$ converges co-ordinatewise; its limit is defined to be $\varphi$. We then define the DTHMC $W = \MC(\upalpha,P)$ using $\varphi$, and show that for any finite sequence $(a_0,b_0), \dots, (a_k,b_k)$, as $m \to \infty$,
\[
    \prob{W^m_0 = \!(a_0,b_0), \dots,\! W^m_k = \!(a_k,b_k)} \to \prob{W_0 = \!(a_0,b_0), \dots,\! W_k = \!(a_k,b_k)}. 
\]
Since we are working in discrete time on a countable state space, this enables us to show that $W^m\overset{d}{\to}W$ as $m\to\infty$. The projections $\proj^1: A \times B \to A$ and $\proj^2: A \times B \to B$ are continuous in the discrete topology, and therefore they induce continuous maps from $(A\times B)^{\mathbb{N}}$ to $A^{\mathbb{N}}$ and $B^{\mathbb{N}}$ with respect to their product topologies. Since $\proj^1(W^m) \eqdist \Xoriginal$ for each $m$ and $W^m \overset{d}{\to} W$, we deduce that $\proj^1(W) \eqdist \Xoriginal$, and similarly $\proj^2(W) \eqdist \Yoriginal$. That is to say, $W$ is a coupling of $\Xoriginal$ and $\Yoriginal$.

In the next two subsections, we define $(W^m)_{m\ge1}$ and calculate the joint distribution of $W^m_0,\dots,W^m_k$ when $m \ge k$. Then in section~\ref{subsec:conditioning} we  look at the particular case of $k=0$ and use it to show that $\varphi=\lim_{m\to\infty}{R^m \mathbf{1}_\Delta}$ is well-defined. We then define $\Delta'$ to be the support of $\varphi$ and show that $\upalpha\in \PM(\Delta')$ and $P\in\SM(\Delta')$.

\subsection{\texorpdfstring{Definition of $W^m$ and $\varphi_m$}{Definition of Wm and phi}}
\label{ss:defW_m}

For each $m \ge 0$, we define a WMC $W^m = (X^m_t, Y^m_t)_{t \ge 0}$, where $X^m \eqdist \Xoriginal$ and $Y^m \eqdist \Yoriginal$. We will define them all on a single probability space. First, let $Z$ be a Markov chain distributed like $\Zoriginal$. Let $X^m_t$ have conditional law equal to that of $\Xoriginal$ given $f(\Xoriginal_0) = Z_0, \dots, f(\Xoriginal_m) = Z_m$. For each $m \ge 0$, let $Y^m$ be conditionally independent of $X^m$ given $Z$, with conditional law equal to that of $\Yoriginal$ given $g(\Yoriginal_0) = Z_0, \dots, g(\Yoriginal_m) = Z_m$.  Observe that by construction $W^m$ takes its values up to time $m$ in $\Delta(f,g)$, but for $t > m$ we may have $f(X^m_t) \neq g(Y^m_t)$.

\begin{remark} 
    Unlike the chain $W^m$, the coupling $W = (X,Y)$ in Theorem~\ref{thm:DTHWMC} will not be defined \emph{a priori} by conditionally independent sampling of $X\eqdist\Xoriginal$ and $Y\eqdist\Yoriginal$ given $f(X)=Z=g(Y)$, but rather it will be defined as $W = (X,Y)=\MC(\upalpha,P)$, where $\upalpha$ and $P$ are given in the statement of Theorem~\ref{thm:DTHWMC}. This definition will make sense once we have proven that $\upalpha \in \PM(\Delta')$ and $P \in \SM(\Delta')$. 
\end{remark}
\begin{remark}
    Although we will not need it as part of the proof of Theorem~\ref{thm:DTHWMC}, it can be shown that $W^m$ is a Markov chain, inhomogeneous in general.  
\end{remark}

For each $m\ge0$, define a column vector $\varphi_m$, indexed by $\Delta$, by
\[
    \varphi_m = R^m \mathbf{1}_\Delta\,.
\]
This is well-defined, because multiplication of infinite matrices with entries in $[0,\infty) \cup \{+\infty\}$ is well-defined and associative\footnote{Multiplication of infinite matrices over $\mathbb{R}$ can fail to be associative, even when every series involved in the computation of the matrix products converges; see \cite{Bossaller2019assoc}.}. Later, in Section~\ref{ss:altdist} we will see that $\varphi_m$ is pointwise bounded on $\Delta$, so all entries of $R^m$ are finite.

We will assume, without loss of generality, that $f$ and $g$ are surjective functions. This is possible because any states in $C$ which are almost surely never visited by $\Zoriginal$ may be discarded from $C$; the hypotheses of Theorem~\ref{thm:DTHWMC} will still be satisfied, and the conclusions are not affected.

\subsection{\texorpdfstring{Temporary assumptions on the support of \ensuremath{\upalpha_X} and \ensuremath{\upalpha_Y}}{Temporary assumptions}}
\label{ss:assumptions}

In Sections~\ref{ss:jointdistofW}
and~\ref{subsec:conditioning}, 
we show that $\varphi$ exists under the following additional assumption before showing that the assumption may be dropped:
\begin{assump} 
\label{assump:fullsupport} 
    $\supp(\upalpha_X) = A$ and  
    $\supp(\upalpha_Y) = B$.
\end{assump} 
Assuming $f$ and $g$ are surjective, it follows that $\supp(\upalpha_Z)=C$.

Having done this, we explain how to modify the argument to deal with the general case, by constructing alternative initial distributions $\upalpha'_X$, $\upalpha'_Y$, and $\upalpha'_Z$ with the following properties:
\begin{itemize}
    \item $\supp(\upalpha'_X) = A$, $\supp(\upalpha'_Y) = B$, and $\supp(\upalpha'_Z) = C$,
    \item $f_*(\upalpha'_X) = g_*(\upalpha'_Y) = \upalpha'_Z$,
    \item $f(\MC(\upalpha'_X,P_X)) \eqdist \MC(\upalpha'_Z,P_Z) \eqdist g(\MC(\upalpha'_Y,P_Y))$.
\end{itemize}

\subsection{\texorpdfstring{The joint distribution of \ensuremath{W^m_0, \dots, W^m_k} when \ensuremath{m \ge k}}{Joint dist of Wm0 to Wmk}}{\,}
\label{ss:jointdistofW}

\begin{lem}
\label{lem: joint dist of Wm}
    Let $k \ge 0$ and $(a_0,b_0), \dots, (a_k,b_k) \in \Delta$ with $c_i = f(a_i) = g(b_i)$ for $0 \le i \le k$. Suppose $\upalpha_Z(c_0) > 0$. Then for each $m \ge k$, 
    \begin{align}
    \begin{split}
    \label{eqn: Wm joint}
    \mathbb{P}&(W^m_0 = (a_0,b_0), \dots, W^m_k = (a_k,b_k)) \\  
        &   
         = \mathbb{E}\bigl[
        \mathbb{P}(X^m_{\le k} = (a_0, \dots, a_k) \mid Z)\, \mathbb{P}(Y^m_{\le k} = (b_0, \dots, b_k) \mid Z)
        \bigr] 
        \\
        & = 
        \mathbb{E}\bigl[
        \mathbb{P}(X^m_{\le k} = (a_0, \dots, a_k) \mid Z_{\le m})\, \mathbb{P}(Y^m_{\le k} = (b_0, \dots, b_k) \mid Z_{\le m})
        \bigr]
        \\
        & = 
        \frac{\upalpha_X(a_0)\upalpha_Y(b_0)}{\upalpha_Z(c_0)}\biggl(\prod_{i= 1}^{k}R((a_{i-1},b_{i-1}),(a_i,b_i))\bigg)\,\varphi_{m-k}(a_k,b_k)\,.
    \end{split}
    \end{align}
\end{lem}

\begin{remark}
    Before proving this lemma, let us explain how it will be used. Once we have proven that $P$ is a stochastic matrix and that $\upalpha$ is a probability distribution, we will know that $\MC(\upalpha,P)$ is a well-defined DTHMC, and
    then with the same assumptions as in Lemma~\ref{lem: joint dist of Wm}, we will have
    \begin{multline*}
    \prob{W_0 = (a_0,b_0), \dots, W_k = (a_k,b_k)} 
     = 
    \upalpha(a_0,b_0) \prod_{i=1}^{k} P((a_{i-1},b_{i-1}), (a_{i},b_{i}))\\
     =  \frac{\upalpha_X(a_0)\upalpha_Y(b_0)}{\upalpha_Z(c_0)}\varphi(a_0,b_0) \prod_{i=1}^k \frac{1}{\varphi(a_{i-1},b_{i-1})}R((a_{i-1},b_{i-1}),(a_i,b_i)) \,\varphi(a_i,b_i)\\
     =  \frac{\upalpha_X(a_0)\upalpha_Y(b_0)}{\upalpha_Z(c_0)}\biggl(\prod_{i=1}^k R((a_{i-1},b_{i-1}),(a_i,b_i))\biggr) \,\varphi(a_k,b_k).
    \end{multline*}
    Once we have proven that $\varphi_m\to\varphi$, we will be able to  take the limit as $m\to\infty$ in~\eqref{eqn: Wm joint} for any $(a_0,b_0), \dots, (a_k,b_k) \in \Delta'$, to obtain
    \begin{multline}
    \label{eqn: convergence of fdds}
        \lim_{m \to \infty} \mathbb{P}(W^m_0 = (a_0,b_0), \dots, W^m_k = (a_k,b_k)) \\
        = \mathbb{P}(W_0 = (a_0,b_0), \dots, W_k = (a_k,b_k))\,.
    \end{multline}
    This also holds in the case where $\upalpha_Z(c_0) = 0$, since then all the probabilities in~\eqref{eqn: convergence of fdds} are $0$.
\end{remark}

\begin{proof}[Proof of Lemma~\ref{lem: joint dist of Wm}]
    For notational convenience, define 
    \begin{align*}
      \mathcal{T} =   \{& ((a_{k+1},b_{k+1},c_{k+1}), \dots, (a_m,b_m,c_m)) \,:\,\prob{Z_{\le m} = (c_0, \dots, c_{m})} > 0,\\  
        &c_{k+1} = f(a_{k+1}) = g(b_{k+1}), \dots, c_m = f(a_m)=g(b_m) \}.
    \end{align*}
    By definition, $\mathbb{P}(W^m_0 = (a_0,b_0), \dots, W^m_k = (a_k,b_k))$ is
    \begin{align*} 
        \sum_\mathcal{T}\mathbb{P}&(Z_{\le m} = (c_0, \dots, c_{m})) \;\times
        \mathbb{P}(X^m_{\le m} = (a_0, \dots, a_{m}) \mid Z_{\le m} = (c_0, \dots, c_{m}))\; 
        \\ &
        \times \mathbb{P}(Y^m_{\le m} = (b_0, \dots, b_{m}) \mid Z_{\le m} = (c_0, \dots, c_{m}))\,.
    \end{align*}
    We rewrite this as a sum over trajectories $(a_{k+1},b_{k+1}), \dots (a_m,b_m)$ in $\Delta$, and express it in terms of the initial distributions and transition matrices of $\Xoriginal, \Yoriginal$, and $\Zoriginal$.
    \begin{equation}
        \sum_\mathcal{T}  \upalpha_Z(c_0)\prod_{i=1}^{m}P_Z(c_{i-1},c_i) \times 
        \frac{\upalpha_X(a_0)}{\upalpha_Z(c_0)} \prod_{i=1}^{m} \frac{P_X(a_{i-1},a_i)}{P_Z(c_{i-1},c_i)} \times 
        \frac{\upalpha_Y(b_0)}{\upalpha_Z(c_0)} \prod_{i=1}^{m} \frac{P_Y(b_{i-1},b_i)}{P_Z(c_{i-1},c_i)}
    \end{equation}
    For the $k\ge1$ case, we rearrange this by breaking each product of the form $\prod_{i=1}^m$ into two products, of the form $\prod_{i=1}^k$ and $\prod_{i={k+1}}^m$. We can collect factors into terms involving $R$ and we recognize that the summation over $(a_{k+1}, b_{k+1}), \dots, (a_{m-1},b_{m-1})$ performs a matrix product, so we end up with the final expression in~\eqref{eqn: Wm joint}. In particular, when $k=0$, we get
    \begin{equation}\label{eqn: Wm init}\begin{aligned}
         \mathbb{P}(W^m_0 = (a_0,b_0)) 
        & = 
        \mathbb{E}[\mathbb{P}(X^m_0 = a_0 \mid Z)\, \mathbb{P}(Y^m_0 = b_0 \mid Z)]\\
        & = 
        \frac{\upalpha_X(a_0)\upalpha_Y(b_0)}{\upalpha_Z(c_0)}\varphi_m(a_0,b_0)
        .
    \end{aligned}\qedhere\end{equation}
\end{proof}

\subsection{\texorpdfstring{Conditioning on \ensuremath{(Z_0, \dots, Z_m)} to prove convergence of \ensuremath{R^m \mathbf{1}_\Delta}}{Proving convergence of phim}}
\label{subsec:conditioning}

For any $m,k \ge 0$ and $a_0, \dots, a_k \in A$, define the random variable
\begin{align*}
    L_m(a_0, \dots, &a_k)   =  \mathbb{P}(X^m_0 = a_0, \dots, X^m_k = a_k \mid Z_0, \dots Z_m) \\ & =  \mathbb{P}(\Xoriginal_0 = a_0, \dots, \Xoriginal_k = a_k \mid f(\Xoriginal_0)=Z_0, \dots f(\Xoriginal_m)=Z_m)\\
    & =  \frac{\mathbb{P}(\Xoriginal_0 = a_0, \dots, \Xoriginal_k = a_k,\; f(\Xoriginal_0)=Z_0, \dots f(\Xoriginal_m)=Z_m)}{\mathbb{P}( f(\Xoriginal_0)=Z_0, \dots f(\Xoriginal_m)=Z_m)}\,.
\end{align*}
Similarly for $b_0, \dots, b_k \in B$ define 
\[
    M_m(b_0, \dots, b_k) = \mathbb{P}(Y^m_0 = b_0, \dots, Y^m_k = b_k \mid Z_0, \dots, Z_m).
\]
Let $\mu_Z$ denote the law of $Z$. The random variables $L_m(a_0, \dots, a_k)$ are a countable family of functions of the random variable $Z$, each of which is defined $\mu_Z$-a.e.~by the formula above, because $C^m$ is countable. The same applies to the random variables $M_m(b_0, \dots, b_k)$. Hence they are \emph{simultaneously} defined $\mu_Z$-a.e., for all finite strings $(a_0, \dots, a_k)$ and $(b_0, \dots, b_k)$, and all $m \in \mathbb{N}$.
    
The sequences $(L_m(a_0, \dots, a_k))_{m \ge 0}$ and $(M_m(b_0, \dots, b_k))_{m \ge 0}$ are (countably many) bounded martingales with respect to the filtration generated by $Z$, so $\mu_Z$-a.e.~they all converge $\mu_Z$ to (random) limits, which we call $L_\infty(a_0, \dots, a_k)$ and $M_\infty(b_0, \dots, b_k)$.  Note that 
\[ 
\sum_{a_0, \dots, a_k \in A} L_m(a_0, \dots, a_k) = 1 \;\text{ a.s.,}
\]
and for each $j \ge 1$ and $(a_0, \dots, a_{j-1}) \in A^j$ we have
\[
\sum_{a_j, \dots, a_k \in A} L_m(a_0, \dots, a_k) = L_m(a_0, \dots, a_{j-1}) \;\text{ a.s.}\,. 
\]
Analogous equations hold for $M_m$. The next lemma will be used to show that the same holds with $L_m$ replaced by $L_\infty$ and likewise for $M_m$ replaced by $M_\infty$.
\begin{lem}[The limit of a martingale partition]\label{lem:martingalepartitionofunity}
Let $I$ be a countable set and for each $i \in I$ let $(U_t^{(i)})_{t \ge 0}$ be a non-negative martingale, all with respect to the same filtration $(\mathcal{F}_t)_{t \ge 0}$ on a common probability space. Suppose that $ \sum_{i \in I} U^{(i)}_t   \le 1$ a.s. for every $t \ge 0$. Then almost surely \[\sum_{i \in I} \lim_{t \to \infty} U^{(i)}_t = \lim_{t \to \infty} \sum_{i \in I} U_t^{(i)}.\]
\end{lem}
\begin{proof} 
    By the martingale convergence theorem, $\lim_{t \to \infty} U_t^{(i)}$ exists a.s. for each $i$ and hence for all $i$ almost surely, since $I$ is countable.  We have
    \[ 
        \mathbb{E}\bigg(\sum_{i \in I} U_{t+1}^{(i)} \mid \mathcal{F}_t\bigg) = \sum_{i \in I}\mathbb{E}\big(U_{t+1}^{(i)} \mid \mathcal{F}_t\big) = \sum_{i \in I} U_{t}^{(i)},
    \]
    so $\sum_{i \in I} U_{t}^{(i)}$ is a non-negative martingale and hence it converges a.s. to a finite limit.  
    Since all $U^{(i)}_t \ge 0$, by Fatou's lemma we have 
    \begin{equation}
    \label{eq:Fatouforpartition}
        \sum_{i \in I} \lim_{t \to \infty} U^{(i)}_t \le \lim_{t \to \infty} \sum_{i \in I} U^{(i)}_t \; \text{ a.s.}
    \end{equation}
    Since $0 \le U_t^{(i)} \le 1$ a.s.  by dominated convergence and the martingale property we have
    \begin{equation}
    \label{eq:DC}
        \mathbb{E}\lim_{t \to \infty} U^{(i)}_t  = \lim_{t \to \infty} \mathbb{E}U_t^{(i)} = \lim_{t \to \infty} \mathbb{E}U_0^{(i)} = \mathbb{E}U_0^{(i)}.
    \end{equation}
    Now we find the expectation of the left-hand side of~\eqref{eq:Fatouforpartition}. By Tonelli's theorem and~\eqref{eq:DC}, 
    \[
        \mathbb{E}\sum_{i \in I} \lim_{t \to \infty} U^{(i)}_t  = \sum_{i \in I} \mathbb{E}\lim_{t \to \infty} U^{(i)}_t = \sum_{i \in I} \mathbb{E} U_0^{(i)}.
    \]
    On the other hand, taking the expectation of the right-hand side of~\eqref{eq:Fatouforpartition}, we may apply the dominated convergence theorem, since $\sum_{i \in I} U_t^{(i)} \le 1$ a.s., obtaining
    \[
        \mathbb{E}\lim_{t \to \infty} \sum_{i \in I} U_t^{(i)} = \lim_{t \to \infty} \mathbb{E}\sum_{i \in I} U_t^{(i)} = \lim_{t \to \infty} \sum_{i \in I}\mathbb{E} U_t^{(i)} = \lim_{t \to \infty} \sum_{i \in I}\mathbb{E} U_0^{(i)} = \sum_{i \in I}\mathbb{E} U_0^{(i)}.
    \]
    Thus, the two non-negative sides of~\eqref{eq:Fatouforpartition} have equal expectation; moreover this expectation is finite: 
     \[ 
        \sum_{i \in I}\mathbb{E} U_0^{(i)} = \mathbb{E}\sum_{i \in I} U_0^{(i)} \le 1.
    \]
    Hence we must have equality almost surely in~\eqref{eq:Fatouforpartition}, as required.
\end{proof}
\begin{cor}
\label{cor:Linftyisaprob}
For $0 \le j \le k$, $(a_0, \dots, a_{j-1}) \in A^{j}$ and $(b_0, \dots, b_{j-1}) \in B^{j}$, we have almost surely
\[ \sum_{a_{j}, \dots, a_k \in A} L_\infty(a_0, \dots, a_k) = \begin{cases} L_\infty(a_0, \dots, a_{j-1}) & \text{ if $j \ge 1$,}\\
1 & \text{ if $j=0$,}\end{cases}\]
and
\[ 
    \sum_{b_{j}, \dots, b_k \in B} M_\infty(b_0, \dots, b_k) = \begin{cases} M_\infty(b_0, \dots, b_{j-1}) & \text{ if $j \ge 1$,}\\  1 & \text{ if $j=0$.}\end{cases}
\]
\end{cor}
\begin{lem}\label{lem: varphim converges}
    Under Assumption~\ref{assump:fullsupport},  the sequence $\varphi_m(a,b)$ converges to a finite limit as $m \to \infty$, for every $(a,b) \in \Delta$.
\end{lem}

\begin{proof}
We have $L_m(a) M_m(b) \to L_\infty(a) M_\infty(b)$ almost surely as $m \to \infty$. Since $L_m(a) M_m(b) \in [0,1]$ a.s. this implies that \[\mathbb{E}(L_m(a) M_m(b)) \to \mathbb{E}(L_\infty(a) M_\infty(b))\,.\] Therefore, for each $(a,b) \in \Delta$, as $m\to\infty$,
\[
    \mathbb{P}(W^m_0 = (a,b)) = \mathbb{E}(\prob{W^m_0 = (a,b) | Z}) \;\to\; \mathbb{E}[L_\infty(a) M_\infty(b)] \,.
\]
Then using~\eqref{eqn: Wm init}, with $c = f(a) = g(b)$, we have as $m \to \infty$
\begin{multline} 
    \frac{\upalpha_X(a)\upalpha_Y(b)}{\upalpha_Z(c)} \varphi_m(a,b) = \mathbb{P}(W_0^m = (a,b))\\ = \mathbb{E}(L_m(a)M_m(b)) \to \mathbb{E}[L_\infty(a) M_\infty(b)]\,.\label{eq:alphaab}
\end{multline} 
This implies that $\varphi_m(a,b)$ converges as $m \to \infty$, since by Assumption~\ref{assump:fullsupport} we have 
$\upalpha_X(a)\upalpha_Y(b)/\upalpha_Z(c) \neq 0$.
\end{proof}
    To show that $\varphi_m(a,b)$ converges without using Assumption~\ref{assump:fullsupport}, we will need to construct alternative initial distributions $\upalpha'_X$, $\upalpha'_Y$, and $\upalpha'_Z$ which do satisfy Assumption~\ref{assump:fullsupport}, and apply Lemma~\ref{lem: varphim converges} to them. This will be done in the next section.

\subsection{Alternative initial distributions}
\label{ss:altdist}

Suppose now that $\upalpha_X, \upalpha_Y$, and $\upalpha_Z$ do not satisfy Assumption~\ref{assump:fullsupport}. Instead, we construct the following alternative initial distributions:
\[
    \upalpha'_X = \sum_{n=0}^\infty 2^{-n-1} \upalpha_X P_X^n\,,
\quad
    \upalpha'_Y = \sum_{n=0}^\infty 2^{-n-1} \upalpha_Y P_Y^n\,,
\quad
    \upalpha'_Z = \sum_{n=0}^\infty 2^{-n-1} \upalpha_Z P_Z^n\,.
\]
These have the following required properties:
\begin{itemize}
    \item $\supp(\upalpha'_X) = A$, $\supp(\upalpha'_Y) = B$, and $\supp(\upalpha'_Z) = C$,
    \item $f_*(\upalpha'_X) = g_*(\upalpha'_Y) = \upalpha'_Z$,
    \item $f(\MC(\upalpha'_X,P_X)) \eqdist \MC(\upalpha'_Z,P_Z) \eqdist g(\MC(\upalpha'_Y,P_Y))$.
\end{itemize}
The first two properties in the list above are straightforward to check. For the third, observe that if $\Xoriginal = \MC(\upalpha_X, P_X)$ and $\Zoriginal = f(\Xoriginal)$, then 
\[
    \MC(\upalpha'_X, P_X) \eqdist (\Xoriginal_{t+G})_{t \ge 0}\,,
\] 
where $G$ is a Geometric($1/2$) random variable taking values in $\{0,1,2,\dots\}$ and independent of $\Xoriginal$. Hence 
\[
    f(\MC(\upalpha'_X, P_X)) \eqdist (f(\Xoriginal_{t+G}))_{t \ge 0} = (\Zoriginal_{t+G})_{t \ge 0}\,.
\]
Since $G$ is independent of $\Xoriginal$, it is also independent of $\Zoriginal$, so 
\[
    (\Zoriginal_{t+G})_{t \ge 0} \eqdist \MC(\upalpha'_Z,P_Z)\,.
\]

Let $\tilde{X}=\MC(\upalpha_X',P_X)$, $\tilde{Y}=\MC(\upalpha'_Y, P_Y)$, and $\tilde{Z}=\MC(\upalpha'_Z, P_Z)$. Define $\tilde{W}^m$ in the same way as $W^m$ except with $\tilde{X}$, $\tilde{Y}$, and $\tilde{Z}$ instead of $\Xoriginal$, $\Yoriginal$, and $\Zoriginal$. Then following the same steps as above, we arrive at
\begin{equation}\label{eq:modifiedprobs} \mathbb{P}(\tilde{W}^m_0 = (a,b)) 
=\frac{\upalpha'_X(a_0)\upalpha'_Y(b_0)}{\upalpha'_Z(c_0)}\varphi_m(a_0,b_0)
    \,.
\end{equation}
Note that $\varphi_m$ does not depend on the choice of initial distributions, since it is given by $\varphi_m = R^m \mathbf{1}_\Delta$.
We may now apply Lemma~\ref{lem: varphim converges} to this modified setup, which does satisfy Assumption~\ref{assump:fullsupport},
to conclude that $\varphi_m(a,b)$ converges to a finite limit as $m \to \infty$ for every $(a,b) \in \Delta$.

We may now define $\varphi: \Delta \to [0,\infty)$ to be the coordinatewise limit: 
\begin{equation}
\label{eqn:varphi def}
    \varphi = \lim_{m \to \infty} \varphi_m = \lim_{m \to \infty} R^m \mathbf{1}_\Delta.
\end{equation} 
Using the fact the left-hand side of equation~\ref{eq:modifiedprobs} is a probability and therefore less than or equal to $1$, we obtain an upper bound on $\varphi_m(a,b)$ for every $m$, and hence an upper bound for $\varphi(a,b)$:
\[ \varphi(a,b) \le \frac{\upalpha_Z'(c)}{\upalpha_X'(a)\upalpha_Y'(b)},\]
where $c = f(a) = g(b)$ as usual. Another observation that will be useful is that if $\varphi(a,b) = 0$ and $R((a,b),(a',b')) > 0$ then
\begin{eqnarray*} 0 \ge \varphi(a,b) & = & \lim_{m \to \infty} \sum_{(c,d) \in \Delta} R^m((a,b),(c,d)) \\ 
&\ge& R((a,b),(a',b')) \lim_{m \to \infty} \sum_{(c,d) \in \Delta} R^{m-1}((a',b'),(c,d))\\
&=& R((a,b),(a',b')) \,\varphi(a',b'), 
\end{eqnarray*}
 so $\varphi(a',b') = 0$ also. 
 
Looking at the definition of $\upalpha$ in equation~\eqref{eq: coupled pi def}, and using~\eqref{eq:alphaab}, we see that for $(a,b) \in \Delta$ with $f(a) = g(b) = c$ we have
\begin{equation}\label{eq:alphaab_as_expectation}
\upalpha(a,b) = \frac{\upalpha_X(a)\upalpha_Y(b)}{\upalpha_Z(c)}\varphi(a,b) = \mathbb{E}[L_\infty(a)M_\infty(b)].
\end{equation}
Note that if $\upalpha_Z(c) = 0$, since then $L_0(a)  = \mathbb{P}(\Xoriginal_0 = a \mid f(\Xoriginal_0) = Z_0) = 0$ a.s. so in this case the expectation in~\eqref{eq:alphaab_as_expectation} is $0$, agreeing with~\eqref{eq: coupled pi def}. We now extend $\upalpha$ to all of $A \times B$ by letting $\upalpha(a,b) = 0$ for $(a,b) \in (A \times B) \setminus \Delta$, and then~\eqref{eq:alphaab_as_expectation} holds on all of $A \times B$, since for $(a,b) \in (A \times B) \setminus \Delta$ we have $L_m(a)M_m(b) = 0$ a.s. for every $m$, so $L_\infty(a)M_\infty(b) = 0$ a.s.
\begin{remark} We will show in Proposition~\ref{prop: cond laws are martingale limits} below that almost surely 
\[L_\infty(a) = \mathbb{P}(\Xoriginal_0 = a \mid f(\Xoriginal) = \Zcopy)\] and \[M_\infty(b) = \mathbb{P}(\Yoriginal_0 = b \mid g(\Yoriginal) = \Zcopy).\]
(The proof of Proposition~\ref{prop: cond laws are martingale limits} is quite technical, so we defer it until Section~\ref{ss:condindep}, near the end of the proof.)
It follows that we have
\begin{eqnarray}
\label{eqn: pi as expectation}
    \upalpha(a,b) & = & \notag
    \frac{\upalpha_X(a)\upalpha_Y(b)}{\upalpha_Z(c)}\varphi(a,b)\\
    &=&
    \mathbb{E}\big[\mathbb{P}(\Xoriginal_0 = a \mid f(\Xoriginal)=\Zcopy) \mathbb{P}(\Yoriginal_0 = b \mid g(\Yoriginal)=\Zcopy)\big]
    \,.
\end{eqnarray}
Since
\[ \frac{\upalpha_X(a)\upalpha_Y(b)}{\upalpha_Z(c)} = \mathbb{E}[\mathbb{P}(\Xoriginal_0 = a \mid f(\Xoriginal_0)=Z_0) \mathbb{P}(\Yoriginal_0 = b \mid g(\Yoriginal_0)=Z_0)],\]
we can write down an expression for $\varphi$ in terms of conditional probabilities, with no limits:
\[
    \varphi(a,b) =
    \frac{
    \mathbb{E}[\mathbb{P}(\Xoriginal_0 = a \mid f(\Xoriginal)=Z) \mathbb{P}(\Yoriginal_0 = b \mid g(\Yoriginal)=Z)]
    }{
    \mathbb{E}[\mathbb{P}(\Xoriginal_0 = a \mid f(\Xoriginal_0)=Z_0) \mathbb{P}(\Yoriginal_0 = b \mid g(\Yoriginal_0)=Z_0)]
    }.
\]
\end{remark}
 This shows that $\varphi$ quantifies the distinction between conditioning only on $Z_0$ and conditioning on the whole trajectory of $Z$.
\subsection{\texorpdfstring{$\upalpha$ is a coupling of $\upalpha_X$ and $\upalpha_Y$}{pi is a coupling of piX and piY}}
\label{ss: pi is pm}
$\upalpha$ takes non-zero values only on the set $\Delta'$ defined by
 \[\Delta' = \{(a,b) \in \Delta\,:\, \varphi(a,b) > 0\}.\]
\begin{lem}
\label{lem: pi is a coupling}
    $\upalpha \in \PM(\Delta')$ and $\proj^1_*(\upalpha) = \upalpha_X$, $\proj^2_*(\upalpha) = \upalpha_Y$.
\end{lem}
\begin{proof}
    Sum equation~\eqref{eq:alphaab_as_expectation} over all $(a,b) \in A \times B$ and apply Tonelli's theorem twice and then Corollary~\ref{cor:Linftyisaprob} to obtain
    \begin{eqnarray*}\sum_{(a,b) \in A \times B} \upalpha(a,b) & = &
    \sum_{(a,b) \in A \times B} \mathbb{E}\big[L_\infty(a)M_\infty(b)\big]\\
    & = &
    \mathbb{E}\bigg[\sum_{(a,b) \in A \times B}L_\infty(a)M_\infty(b) \bigg]\\
    & = & 
    \mathbb{E}\bigg[\sum_{a \in A} L_\infty(a) \sum_{b \in B} M_\infty(b)\bigg] = \mathbb{E}(1) = 1\,.
    \end{eqnarray*}
    Therefore $\upalpha \in \PM(\Delta')$. If instead of summing over all $(a,b) \in \Delta$, we fix $a \in A$ and sum only over $b \in B$, we obtain by a similar argument that
    \begin{eqnarray*}
        \sum_{b \in B} \upalpha(a,b) 
        & = &
        \sum_{b \in B}\mathbb{E}\big[L_\infty(a)M_\infty(b)\big]
        \\
        & = & 
        \mathbb{E}\big[L_\infty(a)\sum_{b \in B}M_\infty(b) \big]
        \\
        & = & 
        \mathbb{E}(L_\infty(a)) = \mathbb{P}(\Xoriginal_0 = a) =  \upalpha_X(a).
    \end{eqnarray*}
    That is, $\proj^1_*(\upalpha) = \upalpha_X$. Likewise, $\proj^2_*(\upalpha) = \upalpha_Y$.
\end{proof}

\subsection{\texorpdfstring{$P$ is stochastic and $\varphi$ is a right eigenvector of $R$}{P is stochastic}} 
\label{ss: P is stochastic}
Recall that $P((a,b),(a',b'))$ is defined only for $(a,b),(a',b') \in \Delta'$, by equation~\eqref{eq: coupled P def}. Note that $P$ does not depend in any way on the choice of initial distributions.
By definition of $P$, the statement that $\varphi$ is a right eigenvector of $R$ for the eigenvalue $1$ is equivalent to the statement that $P$ is a stochastic matrix. 
\begin{lem}
\label{lem: P is stochastic}
    $P$ is stochastic.
\end{lem}
\begin{proof}
    We mimic the proof of Lemma~\ref{lem: pi is a coupling}. First extend the rows of $P$ by zeros if necessary so that its columns are indexed by $\Delta$ rather than $\Delta'$. Fix $(a,b) \in \Delta'$. Since $P$ does not depend on the initial distributions we may assume (by switching to $\upalpha'_X, \upalpha'_Y$ and $\upalpha'_Z$ if necessary) that Assumption~\ref{assump:fullsupport} holds, and therefore $\upalpha(a,b) > 0$. Specializing Lemma~\ref{lem: joint dist of Wm} to $k=1$, we have for any $(a',b') \in \Delta'$ that
    \begin{eqnarray*}
    \mathbb{P}(W^m_0 = (a,b), W^m_1 = (a',b')) \! &=&
    \!\frac{\upalpha_X(a)\upalpha_Y(b)}{\upalpha_Z(c)}R((a,b),(a',b'))\,\varphi_{m-1}(a',b')
    \\
    & = & 
    \!\mathbb{E}[L_m(a,a')M_m(b,b')].
    \end{eqnarray*}
    Taking the limit as $m\to\infty$, we get
    \begin{eqnarray}\notag
        \mathbb{P}(W_0 = (a,b), W_1 = (a',b'))  =  \upalpha(a,b)P((a,b),(a',b'))\\
        = \mathbb{E}(L_\infty(a,a')M_\infty(b,b')),\label{eq:transition}
    \end{eqnarray}
    and this is consistent with our extension of $P$ since both sides above are $0$ if $\varphi(a',b') = 0$. Now sum equation~\eqref{eq:transition} over $(a',b') \in \Delta$, use Tonelli's theorem to swap the expectation and the sum, change to a sum over all $(a',b') \in A \times B$, using the fact that $L_\infty(a,a')M_\infty(b,b') = 0$ if $f(a') \neq g(b')$, separate the sum over $a'$ and the sum over $b'$, and use Corollary~\ref{cor:Linftyisaprob}, to get
    \begin{eqnarray*}
        \upalpha(a,b) \sum_{(a',b') \in \Delta} P((a,b),(a',b')) 
         &=& 
        \sum_{(a',b') \in \Delta}\mathbb{E}[L_\infty(a,a')M_\infty(b,b')]
         \\ 
         &=& \mathbb{E}\bigg[\sum_{(a',b') \in \Delta} L_\infty(a,a')M_\infty(b,b')\bigg]\\
         &=& \mathbb{E}\bigg[\sum_{a' \in A} L_\infty(a,a')\sum_{b' \in B}M_\infty(b,b')\bigg]\\
         &=&
         \mathbb{E}[L_\infty(a)M_\infty(b)]
         \\
         & = &   \upalpha(a,b).
    \end{eqnarray*}
    Since $(a,b) \in \Delta'$, we have $\upalpha(a,b) > 0$ so we may cancel $\upalpha(a,b)$ from both sides to conclude that $P$ is stochastic.
\end{proof}

\subsection{The marginals have the correct laws}
\label{ss: completing the proof}

We now know that $W=\MC(\upalpha,P)$ is a well-defined DTHMC on the state space $\Delta'$, since $\upalpha\in\PM(\Delta')$ and $P\in\SM(\Delta')$.

\begin{lem}
 $W$ is a coupling of the given marginal chains: $\proj^1(W) \eqdist \Xoriginal$ and $\proj^2(W) \eqdist \Yoriginal$.
\end{lem}
\begin{proof}
Since the marginal processes $\proj^1(W)$ and $\proj^2(W)$ take their values in $A^\mathbb{N}$ and $B^\mathbb{N}$ with the product topologies, it suffices to check that their finite-dimensional distributions agree with those of $\Xoriginal$ and $\Yoriginal$ respectively. Let $k \ge 0$ and let $a_0, \dots, a_k \in A$. Combining equations~\eqref{eqn: Wm joint} and~\eqref{eqn: convergence of fdds}, and using the dominated convergence theorem, we obtain
\begin{equation}
\begin{aligned}
\mathbb{P}(W_0 = (a_0,b_0),& \dots, W_k = (a_k,b_k))  \\
 =  \lim_{m \to \infty}& \mathbb{E}[L_m(a_0, \dots, a_k) M_m(b_0, \dots, b_k)] \\  
 &=   \mathbb{E}[L_\infty(a_0, \dots, a_k) M_\infty(b_0, \dots, b_k)].
\label{eqn: W fdds in terms of Linfty and Minfty}
\end{aligned}
\end{equation}
Sum over all $b_0, \dots, b_k \in B$, then use Tonelli and Corollary~\ref{cor:Linftyisaprob} to obtain
\begin{eqnarray}
\mathbb{P}(X_0 = a_0, \dots, X_k = a_k) 
& = &
\sum_{b_0, \dots, b_k \in B} \mathbb{P}(W_0 = (a_0,b_0), \dots, W_k = (a_k,b_k))  \notag \\
& = &
\sum_{b_0, \dots, b_k \in B} \mathbb{E}[L_\infty(a_0, \dots, a_k) M_\infty(b_0, \dots, b_k)] \notag \\
& = &
\mathbb{E}\bigl[L_\infty(a_0, \dots, a_k)\sum_{b_0, \dots, b_k \in B}M_\infty(b_0, \dots, b_k)\bigr] \notag \\
& = & \mathbb{E}(L_\infty(a_0, \dots, a_k)) \,.
\label{eqn: X fdds in terms of Linfty}
\end{eqnarray}
Using dominated convergence,
\[ \mathbb{E}(L_\infty(a_0, \dots, a_k)) 
 =  \lim_{m \to \infty} \mathbb{E}(L_m(a_0, \dots, a_k)) \]
and from the definition of $L_m$ we have for every $m$ that
\begin{eqnarray*} \mathbb{E}(L_m(a_0, \dots, a_k))  & = & \mathbb{E}(\mathbb{P}(\Xoriginal_0 = a_0, \dots, \Xoriginal_k = a_k \mid Z_1, \dots, Z_m)\\
& = & \mathbb{P}(\Xoriginal_0 = a_0, \dots, \Xoriginal_k = a_k).
\end{eqnarray*}
Likewise, for any $b_0, \dots, b_k \in B$ we have
\[\mathbb{P}(Y_0 = b_0, \dots, Y_k = b_k) = \mathbb{P}(\Yoriginal_0 = b_0, \dots, \Yoriginal_k = b_k)\,.\qedhere\]
\end{proof}

\subsection{Proof of conditional independence}
\label{ss:condindep}

Our final task in proving Theorem~\ref{thm:DTHWMC} is to prove the conditional independence of $\Xcopy$ and $\Ycopy$ given $\Zcopy$. 

\begin{remark} 
    We have seen that $X^m$ and $Y^m$ are conditionally independent given $Z$, and $Z$ is a (fixed) continuous function of $X^m$ and also of $Y^m$, and that $(X^m,Y^m)$ converges in distribution to $(X,Y)$. The reader may be tempted to think that it should be possible to deduce from these facts alone that $X$ and $Y$ are conditionally independent given $Z$. However, this is not possible. The obstacle is best illustrated by the following simple example.  Let $\mathcal{E} = (\varepsilon_1, \varepsilon_2, \dots)$ be an i.i.d.~sequence of Bernoulli($1/2$) random variables. Then the random variable $U_n = (\varepsilon_n,\mathcal{E})$ is conditionally independent of  itself given $\mathcal{E}$, and $\mathcal{E}$ is a continuous function of $U_n$, indeed the same function for all $n$. However, $(U_n,U_{n})$ converges jointly in distribution to $((\varepsilon,\mathcal{E}),(\varepsilon,\mathcal{E}))$, where $\varepsilon$ is a Bernoulli($1/2$) random variable independent of the sequence $\mathcal{E}$. So conditional independence is lost in the distributional limit. In this general setup, to ensure that conditional independence is preserved in the limit a stronger mode of convergence is needed. For example, one could require \emph{stable convergence} with respect to the sub-$\sigma$-algebra generated by $Z$ (see H\"{a}usler and Luschgy~\cite{stablebook}), or almost sure convergence, (which can in fact be arranged as a consequence of stable convergence using a conditional Skorohod representation theorem due to Pratelli and Rigo~\cite{pratelli2023strong}), or convergence in total variation, which Lauritzen~\cite{lauritzen} proved suffices to preserve conditional independence in the limit. In our setting, convergence in total variation does not hold, and we prefer to avoid taking a detour into stable convergence or having to construct a coupling of all $(X^m,Y^m)$ such that $(X^m,Y^m) \to (X,Y)$ almost surely. Instead, our proof below uses a characterization of the conditional laws of $\Xcopy$ and $\Ycopy$ given $\Zcopy$ in terms of the martingale limits $L_\infty$ and $M_\infty$.
\end{remark}

Let  $h: \Delta \to C$ be the map
given by $h(a,b) = f(a) = g(b)$. Recall that $\mu_Z$ is the law of $Z$ and that $\Zcopy = f(\Xcopy) = g(\Ycopy) = h(W)$. The conditional independence of $X$ and $Y$ given $Z$ is an immediate consequence of the following proposition.

\begin{prop}
\label{prop: cond laws are martingale limits}
 We may define a regular family of conditional laws of $\Xcopy$ given $\Zcopy$ by 
    \[    \mathbb{P}(\Xcopy_0 = a_0, \dots, \Xcopy_k = a_k \mid \Zcopy) = L_\infty(a_0, \dots, a_k)\]
Similarly we may define a regular family of conditional laws of $\Ycopy$ given $\Zcopy$ by
    \[
        \mathbb{P}(\Ycopy_0 = b_0, \dots, \Ycopy_k = b_k \mid Z) = M_\infty(b_0, \dots, b_k)
    \]
and a conditional law of $W$ given $Z$ by
\[
  \mathbb{P}(W_0 = (a_0,b_0), \dots, W_k = (a_k,b_k) \mid Z) = L_\infty(a_0, \dots, a_k)M_\infty(b_0, \dots, b_k)\,.
\]  
This means that the above specifications on cylinder sets are consistent, and the random probability measures that they therefore define on $A^\mathbb{N}$, $B^\mathbb{N}$ and $\Delta^\mathbb{N}$ form regular families of conditional probability measures in the sense of the disintegration theorem. 
\end{prop}

\begin{proof}
First, we give the proof for $\Xcopy$ and $L_\infty(\cdot)$; the proof for $\Ycopy$ and $M_\infty(\cdot)$ is essentially the same, and at the end we will explain a small modification that is needed for establishing the conditional distribution of $W$ given $Z$.

$L_\infty(\underline{a})$ is a function of $Z$ that is defined $\mu_Z$-a.e., for every $\underline{a} = (a_0, \dots, a_k) \in \bigcup_{i=1}^\infty A^i$ simultaneously, since $A$ is countable. Moreover, $L_\infty(\underline{a})$ is measurable with respect to the restriction of the product $\sigma$-algebra on $C^{\mathbb{N}}$ to the support of $\mu_Z$, since it is the pointwise limit of the real functions $L_m(\underline{a})$, each of which is piecewise constant on the countably many cylinder sets where it is defined. Corollary~\ref{cor:Linftyisaprob} says that  
$\mu_Z$-a.e. the random variables $L_\infty(\underline{a})$ for $\underline{a} \in \bigcup_{i=0}^\infty A^i$ form a consistent collection and therefore define a random probability measure $\nu(Z)$ on $A^\mathbb{N}$ that is a measurable function of the random variable $Z$. We must check that these measures form a family of conditional probability measures for $\Xcopy$ given $f(\Xcopy)$.

Next, we check that $\mu_Z$-a.s., $\nu(Z)$ is supported on the random fibre $f^{-1}(Z) =  \{ (a_i)_{i=0}^\infty \in A^\mathbb{N}\,:\, (\forall i) f(a_i) = Z_i\}$. In other words, $L_\infty(a_0, \dots, a_i) = 0$ on the event $Z_i \neq a_i$. This follows since $L_m(a_0, \dots, a_i) = 0$ for all $m \ge i$ on the event $Z_i \neq a_i$.

Next, we check that the random probability measure $\nu(Z)$ disintegrates the law $\mu_X$ of $X$ over its image $Z = f(X)$. That is,  $\int \nu(Z) \,d\mu_Z(Z) = \mu_X$.  By a standard Dynkin $\pi-\lambda$ argument, it suffices to check this on the generating $\pi$-system of cylinder sets, i.e.~to check that for every $k \ge 0$ and $a_0, \dots, a_k \in A$ we have
\[ \mathbb{P}(X_0 = a_0, \dots, X_k = a_k) = \mathbb{E}(L_\infty(a_0, \dots, a_k))\,,\] 
which was already established in equation~\eqref{eqn: X fdds in terms of Linfty}.

The proof for the conditional law of $W$ given $Z$ requires an extra step. Define
\[N_m((a_0,b_0), \dots, (a_k,b_k)) = L_m(a_0, \dots, a_k) M_m(b_0, \dots, b_k)\]
and
\begin{eqnarray*}
N_\infty((a_0,b_0), \dots, (a_k,b_k)) & = & L_\infty(a_0, \dots, a_k) M_\infty(b_0, \dots, b_k)\\
& = & \lim_{m \to \infty} N_m((a_0,b_0), \dots, (a_k,b_k)), \text{ $\mu_Z$-a.s.}
\end{eqnarray*} 
Define \[N_\infty((a_0,b_0), \dots, (a_k,b_k)) = L_\infty(a_0, \dots, a_k) M_\infty(b_0, \dots, b_k)\,.\] By multiplying the equations in Corollary~\ref{cor:Linftyisaprob} we find that for every $0 \le j \le k$, $(a_0, \dots, a_{j-1}) \in A^j$ and $(b_0, \dots, b_{j-1}) \in B^j$ we have almost surely
\begin{align*} \sum_{(a_j,b_j), \dots, (a_k,b_k) \in A \times B} N_\infty((a_0,&b_0), \dots, (a_k,b_k)) \\ &= \begin{cases} N_\infty((a_0,b_0), \dots, (a_{j-1},b_{j-1})) & \text{if $ j \ge 1$,}\\ 
1 & \text{if $j=0$.}\end{cases}
\end{align*}
Therefore $\mu_Z$-a.e.~$N_\infty(\cdot)$ determines a Borel probability measure  $\xi(Z)$ on $(A\times B)^\mathbb{N}$ that depends measurably on $Z$. Note that if for any $k \ge 0$ either $Z_k \neq f(a_k)$ or $Z_k \neq g(b_k)$ then $N_\infty((a_0,b_0), \dots, (a_k,b_k)) = 0$, so
$\xi(Z)$ is supported on $h^{-1}(Z) \subseteq \Delta'^{\mathbb{N}}$. Finally, equation~\eqref{eqn: W fdds in terms of Linfty and Minfty} confirms that the family $\xi(Z)$ is a disintegration of the law of $W$ along $Z = h(W)$.
\end{proof}

\section{Variants and corollaries}
\label{sec:others}
Here we prove Theorem~\ref{thm:stationary} and Corollaries~\ref{thm:DTIWMC}, \ref{cor:quasistationary}, and \ref{cor:DTHWMC (of many processes)}.

\subsection{The inhomogeneous case}
\label{ss:HtoI}
\begin{proof}[Proof of Corollary~\ref{thm:DTIWMC}]
For a sequence of stochastic matrices $\left(P_t\right)_{t \geq 1}$, we use $\MC(\upalpha, \left(P_t\right)_{t \geq 1})$ to denote the inhomogeneous Markov chain with initial distribution $\upalpha$ and successive transition matrices $P_1, P_2, \dots$.

In the setting of Corollary~\ref{thm:DTIWMC}, $\Xoriginal = \MC(\upalpha, \left(P_t\right)_{t \geq 1})$ may be represented as a time-forgetting projection of a homogeneous Markov chain $\widehat{X}$ that has state space $A \times \mathbb{N}$, where the second factor is a time co-ordinate. The initial distribution of $\widehat{X}$ is $\widehat{\upalpha}$, supported on $A \times \{0\}$ and given by \[\widehat{\upalpha}((i,0)) = \upalpha(i) \quad \text{ for every $i \in A$.}\]
The transition matrix of $\widehat{X}$ is
\[ P((i,s),(j,t)) = \delta_{s,t-1} P_t(i,j)\,.\]
Similarly we may represent $\Yoriginal$ and $\Zoriginal$ as time-forgetting projections of homogeneous Markov chains $\widehat{Y}$ and $\widehat{Z}$.
In the setting of Corollary~\ref{thm:DTIWMC} we may define maps $\widehat{f}: A \times \mathbb{N} \to C \times \mathbb{N}$ and $\widehat{g}: B \times \mathbb{N} \to C \times \mathbb{N}$ by 
\[\widehat{f}(i,t) = (f(i),t) \quad \text{ and } \quad \widehat{g}(i,t) = (g(i),t)\,.\]
Then 
 \[
    \bigl(\widehat{f}(\widehat{X}_t)\bigr)_{t \geq 0} \eqdist \bigl(\widehat{g}(\widehat{Y}_t)\bigr)_{t \geq 0} \eqdist \widehat{Z}\,.
\]  
If $\widehat{X}$ and $\widehat{Y}$ are coupled so that $W  = (\widehat{X}_t,\widehat{Y}_t)_{t \ge 0}$ is a Markov chain with state space $(A \times \mathbb{N}) \times (B \times \mathbb{N})$ that satisfies $\widehat{f}(\widehat{X}_t) = \widehat{g}(\widehat{Y}_t)$ for all $t \ge 0$, then by forgetting the (common) time coordinate, i.e. projecting to $A \times B$, we obtain a possibly inhomogeneous Markov chain $(X_t,Y_t)_{t \ge 0}$ that is a coupling of $\Xoriginal$ and $\Yoriginal$ and satisfies $f(X_t) = g(Y_t)$ for all $t \ge 0$.   
\end{proof}

\subsection{The stationary case}

The proof of Theorem~\ref{thm:stationary} is formally similar to that of Theorem~\ref{thm:DTHWMC}. The setup is a little different because now the processes $\Xoriginal = (\Xoriginal_t)_{t \in \mathbb{Z}}$ and $\Yoriginal = (\Yoriginal_t)_{t\in \mathbb{Z}}$ are two-sided. Recall that the stationary distributions $\uppi_X, \uppi_Y$ and $\uppi_Z$ are assumed (without loss of generality) to have full support on $A,B$ and $C$, so now there is no need to introduce alternative distributions as we did in Section~\ref{ss:altdist}. 

We replace the couplings $W^m$ used in the proof of Theorem~\ref{thm:DTHWMC} by a sequence $W^{-m,m}$ of couplings of the two-sided processes, as follows. Let $Z$ be a Markov chain distributed like $\Zoriginal$. For each $m \ge 0$, let $X^{-m,m}_t$ have conditional law equal to that of $\Xoriginal$ given $f(\Xoriginal_{-m}) = Z_{-m}, \dots, f(\Xoriginal_m) = Z_m$. 
For each $m \ge 0$, let $Y^{-m,m}$ be conditionally independent of $X^{-m,m}$ given $Z$, with conditional law equal to that of $\Yoriginal$ given $g(\Yoriginal_{-m}) = Z_{-m}, \dots, g(\Yoriginal_m) = Z_m$. Then let $W^{-m,m} = (X^{-m,m}_t, Y^{-m,m}_t)_{t \ge 0}$. Note that $X^{-m, m} \eqdist \Xoriginal$ and $Y^{-m, m} \eqdist \Yoriginal$, so $W^{-m,m}$ is a coupling of $X$ and $Y$.  

As explained after the statement of Theorem~\ref{thm:stationary}, we let $R^{\rev}$ and $\varphi_m^{\rev}$ be defined like $R$ and $\varphi_m$ but starting with the matrices $P_X^{\rev}$, $P_Y^{\rev}$ and $P_Z^{\rev}$ that are the transition matrices of the time reversals of $\Xoriginal$, $\Yoriginal$ and $\Zoriginal$, in place of $P_X, P_Y$ and $P_Z$. 
That is, 
\[\varphi_m^{\rev} = \bigl(R^{\rev}\bigr)^m\mathbf{1}_{\Delta}.\]
We still have the pointwise limit $\varphi = \lim_{m \to \infty} \varphi^m$, exactly as in the one-sided case, and the same applies to the time-reversed data to show that the pointwise limit $\varphi^{\rev} = \lim_{m \to \infty} \varphi_m^{\rev}$ exists. 
The matrix $P$ is defined exactly as before on $\Delta'$, by equation~\eqref{eq: coupled P def}. 

\begin{lem}
\label{lem: joint dist of W^-m,m}
    Let $k,\ell \in \mathbb{Z}$ with $k \le \ell$. Let $(a_{k},b_{k}), \dots, (a_\ell,b_\ell) \in \Delta$ with $c_i = f(a_i) = g(b_i)$ for $k \le i \le \ell$. Suppose $\mathbb{P}(Z_{k} = c_{k}, \dots, Z_\ell = c_\ell) > 0$. Then for each $m \ge \max(|k|,|\ell|)$, 
    \begin{align}
    \begin{split}
    \label{eqn: W^-m,m joint}
    \mathbb{P}&(W^{-m,m}_{k} = (a_{k},b_{k}), \dots, W^{-m,m}_\ell = (a_\ell,b_\ell)) \\  
        &   
         = \mathbb{E}\Bigl[
        \mathbb{P}(X^{-m,m}_{[k,\ell]} = (a_{k}, \dots, a_\ell) \mid Z)\, \mathbb{P}(Y^{-m,m}_{[k,\ell]} = (b_{k}, \dots, b_\ell) \mid Z)
        \Bigr] 
        \\
        & = 
        \mathbb{E}\Bigl[
        \mathbb{P}(X^{-m,m}_{[k,\ell]} = (a_{k}, \dots, a_\ell) \!\mid\!Z_{[-m,m]})\, \mathbb{P}(Y^{-m,m}_{[k,\ell]} = (b_{k}, \dots, b_\ell) \!\mid \!Z_{[-m,m]})
        \Bigr]
        \\
        & = 
        \frac{\uppi_X(a_{k})\uppi_Y(b_{k})}{\uppi_Z(c_{k})}\varphi^{\rev}_{m+k}(a_{k},b_{k})\,\varphi_{m-\ell}(a_{\ell},b_{\ell})\prod_{i= k+1}^{\ell}R((a_{i-1},b_{i-1}),(a_i,b_i)).
    \end{split}
    \end{align}
\end{lem}

\begin{proof}
    Define 
    \begin{align*}
        \mathcal{T}_+ =   \bigl\{& ((a_{\ell+1},b_{\ell+1},c_{\ell+1}), \dots, (a_m,b_m,c_m)) \,:\,\prob{Z_{[\ell,m]} = (c_{\ell}, \dots, c_{m})} > 0,\\  
        &c_{\ell+1} = f(a_{\ell+1}) = g(b_{\ell+1}), \dots, c_m = f(a_m)=g(b_m) \bigr\}
    \end{align*}
    and
    \begin{align*}
        \mathcal{T}_- =   \bigl\{& ((a_{-m},b_{-m},c_{-m}), \dots, (a_{k-1},b_{k-1},c_{k-1})) \,:
        \\
        &\,\prob{Z_{[-m,k]} = (c_{-m}, \dots, c_{k})} > 0,\\ 
        &c_{-m} = f(a_{-m}) = g(b_{-m}), \dots, c_{k-1} = f(a_{k-1})=g(b_{k-1}) \bigr\}.
    \end{align*}
    
    By definition, $\mathbb{P}(W^{-m,m}_{k} = (a_{k},b_{k}), \dots, W^{-m,m}_\ell = (a_\ell,b_\ell))$ is
    \begin{align*} 
        \sum_{\mathcal{T}_+ \times \mathcal{T}_-}\mathbb{P}&(Z_{[-m,m]} = (c_{-m}, \dots, c_{m})) \\
        &\times
        \mathbb{P}(X^{-m,m}_{[-m,m]} = (a_{-m}, \dots, a_{m}) \mid Z_{[-m,m]} = (c_{-m}, \dots, c_{m}))\; 
        \\ &
        \times \mathbb{P}(Y^{-m,m}_{[-m,m]} = (b_{-m}, \dots, b_{m}) \mid Z_{[-m,m]} = (c_{-m}, \dots, c_{m}))\,.
    \end{align*}
    We rewrite this a summation over $\mathcal{T}_+$ and $\mathcal{T}_-$, and express it in terms of the initial distributions and transition matrices of $\Xoriginal, \Yoriginal$, and $\Zoriginal$. We get
    \begin{align*}
        \sum_{\mathcal{T}_+ \times \mathcal{T}_-}  
        &
        \uppi_Z(c_k)\prod_{i=k+1}^{m}P_Z(c_{i-1},c_i) \prod_{i=1-k}^m P^{\rev}_Z(c_{1-i},c_{-i})
        \\ &\times 
        \frac{\uppi_X(a_k)}{\uppi_Z(c_k)} \prod_{i=k+1}^{m} \frac{P_X(a_{i-1},a_i)}{P_Z(c_{i-1},c_i)}
         \prod_{i=1-k}^{m} \frac{P^{\rev}_X(a_{1-i},a_{-i})}{P^{\rev}_Z(c_{1-i},c_{-i})}\\
        &\times 
        \frac{\uppi_Y(a_k)}{\uppi_Z(c_k)} \prod_{i=k+1}^{m} \frac{P_Y(b_{i-1},b_i)}{P_Z(c_{i-1},c_i)} \prod_{i=1-k}^{m} \frac{P_Y^{\rev}(b_{1-i},b_{-i})}{P_Z^{\rev}(c_{1-i},c_{-i})}
        \,.
    \end{align*}
    We rearrange this by breaking each product of the form $\prod_{i=k+1}^m$ into two products, of the form $\prod_{i=k+1}^{\ell}$ and $\prod_{i={\ell+1}}^m$. We can collect factors into terms involving $R$ or $R^{\rev}$ and we recognize that the summations over $\mathcal{T}_+$ and $\mathcal{T}_-$ perform two matrix products, so we end up with the final expression in~\eqref{eqn: W^-m,m joint}.
    \end{proof}
    In particular, when $k=\ell$, we get
    \begin{align}
    \label{eqn: W^-m,m init}
     \mathbb{P}(W^{-m,m}_k = (a_k,b_k))   \notag
         &= 
        \mathbb{E}[\mathbb{P}(X^{-m,m}_k = a_k \mid Z)\, \mathbb{P}(Y^{-m,m}_k = b_k \mid Z)]\\
         &= \frac{\uppi_X(a_k)\uppi_Y(b_k)}{\uppi_Z(c_k)}\,\varphi_{m-k}(a_k,b_k)\,\varphi^{\rev}_{m+k}(a_k,b_k).
    \end{align}

\begin{proof}[Proof of Theorem~\ref{thm:stationary}]
The proof of Theorem~\ref{thm:stationary} follows from Lemma~\ref{lem: joint dist of W^-m,m} in exactly the same way that the proof of Theorem~\ref{thm:DTHWMC} follows from Lemma~\ref{lem: joint dist of Wm}, using the modified martingales
\[L_{-m,m}^{[k,\ell]}(a_{k}, \dots, a_\ell)   =  \mathbb{P}(X^{-m,m}_{k} = a_{k}, \dots, X^{-m,m}_\ell = a_\ell \mid Z_{-m}, \dots Z_m)\] 
and
\[M_{-m,m}^{[k,\ell]}(b_{k}, \dots, b_\ell)   =  \mathbb{P}(Y^{-m,m}_{k} = b_{k}, \dots, Y^{-m,m}_\ell = b_\ell \mid Z_{-m}, \dots Z_m)\] in place of $L_m(\cdot)$ and $M_m(\cdot)$. 
We have
\[ \mathbb{P}(W_k^{-m,m} = (a,b)) = \mathbb{E}\bigl[L_{-m,m}^{[k,k]}(a)M_{-m,m}^{[k,k]}(b)\bigr]\,.\]
By~\eqref{eqn: W^-m,m init}, for each $(a,b) \in \Delta$ with $f(a) = g(b) = c$, as $m \to \infty$ we have
\[ \mathbb{P}(W_k^{-m,m} = (a,b)) \to \frac{\uppi_X(a)\uppi_Y(b)}{\uppi(c)} \varphi(a,b)\varphi^{\rev}(a,b) =: \uppi(a,b).\]
We claim that $\uppi$ is a probability measure supported on the set
\[\Delta'' = \{(a,b) \in \Delta(f,g) \,:\, \varphi(a,b) > 0, \varphi^{\rev}(a,b) >0\}.\]
This may be proved by an argument analogous to the proof of Lemma~\ref{lem: pi is a coupling}.   

We may take the limit as $m \to \infty$ in~\eqref{eqn: W^-m,m joint}, obtaining
\begin{multline} 
\mathbb{P}(W^{-m,m}_{k} = (a_{k},b_{k}), \dots, W^{-m,m}_\ell = (a_\ell,b_\ell))\\  \to \frac{\uppi_X(a_{k})\uppi_Y(b_{k})}{\uppi_Z(c_{k})}\varphi^{\rev}(a_{k},b_{k})\,\varphi(a_{\ell},b_{\ell})\prod_{i= k+1}^{\ell}R((a_{i-1},b_{i-1}),(a_i,b_i)).
\label{eq:W-mmDistLimit}
\end{multline}
Recall that if $\varphi(a,b) = 0$ and $R((a,b),(a',b')) > 0$ then $\varphi(a',b') = 0$. Hence if the right-hand side of~\eqref{eq:W-mmDistLimit} is strictly positive then $\varphi(a_i,b_i) > 0$ for all $k \le i \le \ell$ and hence~\eqref{eq:W-mmDistLimit} may be rewritten as
\begin{multline} 
\mathbb{P}(W^{-m,m}_{k} = (a_{k},b_{k}), \dots, W^{-m,m}_\ell = (a_\ell,b_\ell))\\  \to \frac{\uppi_X(a_{k})\uppi_Y(b_{k})}{\uppi_Z(c_{k})}\varphi^{\rev}(a_{k},b_{k})\,\varphi(a_{k},b_{k})\prod_{i= k+1}^{\ell}P((a_{i-1},b_{i-1}),(a_i,b_i))\\
 = \uppi(a_k,b_k) \prod_{i=k+1}^\ell P((a_{i-1},b_{i-1}),(a_i,b_i)).
\label{eq:W-mmDistLimitw}
\end{multline}
On the other hand, if any $(a_i,b_i)$ does not lie in $\Delta'$ then 
\[\mathbb{P}(W^{-m,m}_{k} = (a_{k},b_{k}), \dots, W^{-m,m}_\ell = (a_\ell,b_\ell)) \to 0.\]
Comparing the case $k=0, \ell = 1$ with the case $k=1,\ell=1$, we have
\[\mathbb{P}(W^{-m,m}_1 = (a_1,b_1)) = \sum_{(a_0,b_0) \in \Delta} \mathbb{P}\bigl(W^{-m,m}_0 = (a_0,b_0), W^{-m,m}_1 = (a_1,b_1)\bigr) .  \]
Taking the limit as $m \to \infty$ and using Fatou's lemma, we find that for every $(a',b') \in \Delta'$, we have
\[ \uppi(a_1,b_1) \ge  \sum_{(a_0,b_0) \in \Delta'} \uppi(a_0,b_0)\, P((a_0,b_0),(a_1,b_1)) = (\uppi P)(a_1,b_1).
 \]
 Since $\uppi \in \PM(\Delta')$ and $P \in \SM(\Delta')$, we have $\uppi P \in \PM(\Delta')$ so the above inequality can only hold by being equality everywhere.  Hence $\uppi P = \uppi$. In particular, $P$ preserves the support of $\varphi^{\rev}$, so we may restrict $P$ to $\Delta''$ to obtain a stochastic matrix.

We have now shown that $W^{-m,m}$ converges in its finite-dimensional distributions to a stationary two-sided Markov chain $W = MC(\uppi,P)$ taking values in $\Delta''$, whose finite-dimensional distributions of $W$ are given by
\begin{equation}
\mathbb{P}(W_{k} = (a_{k},b_{k}), \dots, W_{\ell} = (a_\ell,b_\ell)) = \uppi(a_k,b_k) \prod_{i=k+1}^\ell P((a_{i-1},b_{i-1}),(a_i,b_i)).  \label{eq:fddstat}
\end{equation}

The proof that $W$ is a coupling of $\Xoriginal$ and $\Yoriginal$ and the proof of the conditional independence of $\Xcopy$ and $\Ycopy$ given $\Zcopy$ are similar to those in the proof of Theorem~\ref{thm:DTHWMC}.

We remark that by applying the same proof to the time-reversed chains $\MC(\uppi_X,P_X^{\rev})$ and $\MC(\uppi_Y,P_Y^{\rev})$, we obtain a stochastic matrix $P^{\rev}$ also defined on $\Delta''$ 
by
\[P^{\rev}((a,b),(a',b')) = \frac{1}{\varphi^{\rev}(a,b)}R^{\rev}((a,b),(a',b')) \,\varphi^{\rev}(a',b')\,.\]
Let us confirm algebraically that $P^{rev}$ is the time-reversal of $P$, which means that the following detailed balance condition holds for all $(a,b),(a',b') \in \Delta''$.  
\begin{equation}
\label{eq:detailedbalancestat}
\uppi(a,b)\,P((a,b),(a',b')) = \uppi(a',b')\,P^{\rev}((a',b'),(a,b)).
\end{equation}
Letting $c = f(a) = g(b)$ and $c' = f(a') = g(b')$ and expanding each side, we must check
\begin{multline*}\frac{\uppi_X(a) \uppi_Y(b)}{\uppi_Z(c)}\,\frac{P_X(a,a')P_Y(b,b')}{P_Z(c,c')} \,\varphi(a',b')\varphi^{\rev}(a,b) \\
=  \frac{\uppi_X(a')\uppi_Y(b')}{\uppi_Z(c')}\,\frac{P^{\rev}_X(a',a)P^{\rev}_Y(b',b)}{P^{\rev}_Z(c',c)}\,\varphi(a',b')\,\varphi^{\rev}(a,b).\end{multline*}
Here the factors $\varphi^{\rev}(a,b)$ and $\varphi(a',b')$ are equal on both sides; if either is zero then we are done, and otherwise we may cancel them. Then~\eqref{eq:detailedbalancestat} follows from the detailed balance equations relating $\Xoriginal$, $\Yoriginal$ and $\Zoriginal$ to their time-reversals, i.e.
\[\uppi_X(a)P_X(a,a') = \uppi_X(a')P^{\rev}_X(a',a),\]
\[\uppi_Y(b)P_Y(b,b') = \uppi_Y(b')P^{\rev}_Y(b',b),\]
\[\uppi_Z(c)P_Z(c,c') = \uppi_Z(c')P^{\rev}_Z(c',c).\qedhere\]
\end{proof}

\subsection{The quasistationary case}\label{ss:quasistationary}
\begin{proof}[Proof of Corollary~\ref{cor:quasistationary}]
Let $\upalpha_X$, $\upalpha_Y$ and $\upalpha_Z$ be the initial quasistationary distributions of $\Xoriginal$, $\Yoriginal$ and $\Zoriginal$. This means that 
\[ \upalpha_X P_X = (1-\lambda) \upalpha_X + \lambda \delta_{\rho_A}\,,\]
\[ \upalpha_Y P_Y = (1-\lambda) \upalpha_Y + \lambda \delta_{\rho_B}\,,\]
\[ \upalpha_Z P_Z = (1-\lambda) \upalpha_Z + \lambda \delta_{\rho_C}\,.\]
Let us define modified transition matrices $\hat{P}_X$, $\hat{P}_Y$ and $\hat{P}_Z$, by modifying only the rows corresponding to the states $\rho_A$, $\rho_B$ and $\rho_C$, so that instead of being absorbing states, these become states from which the chain jumps to a state distributed according to the quasistationary distribution.  That is,
\[ \hat{P}_X(x,y) = \begin{cases} P_X(x,y) & \text{if $x \neq \rho_A$,}\\ \upalpha_X(y) & \text{if $x = \rho_A$,}\end{cases}\]
and similarly for $\hat{P}_Y$ and $\hat{P}_Z$.
 Then $\hat{P}_X$ is irreducible and has unique stationary distribution $\uppi_X$, given by \[\uppi_X = \frac{\lambda}{1+\lambda} \delta_{\rho_A} + \frac{1}{1+\lambda} \upalpha_X\,.\]
 Analogous formulae define $\uppi_Y$ and $\uppi_Z$, the unique stationary distributions of $\hat{P}_Y$ and $\hat{P}_Z$. Moreover, the stationary Markov chain $\MC(\uppi_X,\hat{P}_X)$ lumps weakly under $f$ to $\MC(\uppi_Z,\hat{P}_Z)$, and similarly $\MC(\uppi_Y,\hat{P}_Y)$ lumps weakly under $g$ to $\MC(\uppi_Z,\hat{P}_Z)$. Apply Theorem~\ref{thm:stationary} to obtain a stationary Markov chain $\hat{W} = \MC(\uppi,\hat{P}) = (\hat{X}_t,\hat{Y}_t)_{t \ge 0}$ that couples the modified chains, such that $f(\hat{X}_t) = g(\hat{Y}_t)$ for all $t \ge 0$, and such that the marginal chains $\hat{X}$ and $\hat{Y}$ are conditionally independent given $(f(\hat{X}_t))_{t \ge 0}$. 
 The state space $\Delta'$ of $\hat{W}$ includes the state $(\rho_A,\rho_B)$ since this is the only state of $\Delta(f,g)$ that is a lift of $\rho_C$. Modify $\hat{P}$ to make the state $(\rho_A, \rho_B)$ absorbing, obtaining a transition matrix $P$, and let $\upalpha$ be the normalization of the restriction of $\uppi$ to $\Delta' \setminus \{(\rho_A,\rho_B)\}$.
 Then $W = (X_t,Y_t)_{t \ge 0} = \MC(\upalpha,P)$ has the required properties.  Using the fact that the excursions of $\hat{W}$ from $(\rho_A,\rho_B)$ are independent, one can check that conditional on the trajectory of $f(X_t)$ up to the absorption time, the trajectories of $X$ and $Y$ are independent.   
\end{proof}

\subsection{Coupling multiple chains}\label{ss:proof of cor many}

\begin{proof}[Proof of Corollary~\ref{cor:DTHWMC (of many processes)}]
    The proof is straightforward using induction over $k$. The base case $k=1$ is trivial. For the induction step, suppose the result holds for $k-1$, so that we have a homogeneous Markov chain coupling $(X^{(1)}_t, \dots, X^{(k-1)}_t)_{t \ge 0}$ of $\Xoriginal^{(1)}, \dots, \Xoriginal^{(k-1)}$ such that \[f_1(X^{(1)}_t)=\dots = f_{k-1}(X^{(k-1)}_t) \text{ for all $t\ge0$,}\] and such that the marginal processes $X^{(1)}, \dots, X^{(k-1)}$ are conditionally independent given $f_1(X^{(1)})$. Then we may use Theorem~\ref{thm:DTHWMC} to construct a WMC of the two chains $(X^{(1)},\dots,{X}^{(k-1)})$ and $\Xoriginal^{(k)}$ such that 
    \[  
        f_1(X^{(1)}_t) = f_k(X^{(k)}_t) \text{ for all $t\ge0$}
    \]
    \enlargethispage*{1cm}
    and such that $X^{(1)}, \dots, X^{(k)}$ are conditionally independent given $f_1(X^{(1)})$. This completes the induction step.
\end{proof}

%
%

\section{Examples}
\label{sec:examples}

We give several examples which serve to illustrate key points.  In Example~\ref{ex:nonMarkovZnocoupling}, Markov chains $\Xoriginal$ and $\Yoriginal$ share an image process $\Zoriginal$ that is not itself Markov, and we show that they have no weak Markovian coupling taking values in $\Delta(f,g)$. 
Example~\ref{ex:stationarymarginals} illustrates that when $\Xoriginal$ and $\Yoriginal$ are stationary, the coupling constructed in our proof of Theorem~\ref{thm:DTHWMC} need not be stationary. 
Example~\ref{ex: 3 states each affirmative} is a very simple case where there exists an exact Markovian coupling but no strong Markovian coupling.
In Example~\ref{ex:biasedRW}, we discuss an application of Theorem~\ref{thm:DTHWMC} for countably infinite state spaces, coupling two biased random walks on $\mathbb{Z}$ to have equal absolute values and moreover to be conditionally independent given the entire trajectory of the common absolute value.

\begin{example}[Two Markov chains with a non-Markov image process]
\label{ex:nonMarkovZnocoupling}
    Here we show that the hypothesis in Theorem~\ref{thm:DTHWMC} that the image process $Z$ is Markov cannot be dropped.
    This example is based on an example given by Benjamini, Kozma, Lov\'asz, Romik and Tardos  \cite{benjamini2006waiting} of two different rooted trees which have equal return time distributions for a simple symmetric random walk to return to the root. We quote from their paper:
    
       \textit{
        A spelunker has an accident in the cave. Their lamp goes out, they cannot move, all they can hear is a bat flying by every now and then on its random flight around the cave. What can they learn about the shape of the cave?     
        In other words: What can we learn about the structure of a finite graph using only information obtained by observing the returns of a random walk on the graph to this node?}
        
    The authors show that in general it is not possible to determine the structure of the rooted graph from the return time distribution. (They also show that the return time distribution determines the spectrum of the graph Laplacian but not the multiplicities of the eigenvalues.) They construct pairs of rooted graphs that are not isomorphic but have the same distribution of the return time of simple symmetric random walk to the root, one example of which is the pair of trees $T_1$ and $T_2$ shown in Figure~\ref{fig:Benjamini_trees}.

    \begin{figure}
    \centering
    
    \begin{tikzpicture}
    
        \begin{scope}[every node/.style={circle,draw,inner sep=2pt}]
    
            \node (A) at (1,2) {} ;
            \node (B) at (2,2) {} ;
            \node (C) at (3,2) {} ;
            \node (D) at (0,1) {} ;
            \node (E) at (1,1) {} ;
            \node [fill=black!40] (F) at (2,1) {} ;
            \node (G) at (3,1) {} ;
            \node (H) at (4,1) {} ;
            \node (I) at (1,0) {} ;
            \node (K) at (3,0) {} ;
            
            \node (A') at (5,0)  {} ;
            \node (B') at (6,0)  {} ;
            \node (C') at (7,0)  {} ;
            \node[fill=black!40] (D') at (8,0)  {} ;
            \node (E') at (9,0)  {} ;
            \node (F') at (10,0) {} ;
            \node (G') at (11,0) {} ;
            \node (H') at (8,1)  {} ;
            \node (I') at (8,2)  {} ;
            \node (J') at (8,3)  {} ;
            
        \end{scope}
        
        \begin{scope}[>={Stealth}, every edge/.style={draw=black,thick}]
        
            \path (A) edge (E);
            \path (B) edge (F);
            \path (C) edge (G);
            
            \path (D) edge (E);
            \path (E) edge (F);
            \path (F) edge (G);
            \path (G) edge (H);
    
            \path (I) edge (E);
          
            \path (K) edge (G);
    
    
            \path (A') edge (B');
            \path (B') edge (C');
            \path (C') edge (D');
            \path (D') edge (E');
            \path (E') edge (F');
            \path (F') edge (G');
            
            \path (D') edge (H');
            \path (H') edge (I');
            \path (I') edge (J');

        \end{scope}
    
        \begin{scope}
            \node at (2,-1) {$T_1$};
            \node at (8,-1) {$T_2$};
        \end{scope}
    
    \end{tikzpicture}\\
    
    \caption{Bat-cave example from \cite{benjamini2006waiting}.}
    \label{fig:Benjamini_trees}
    \end{figure}
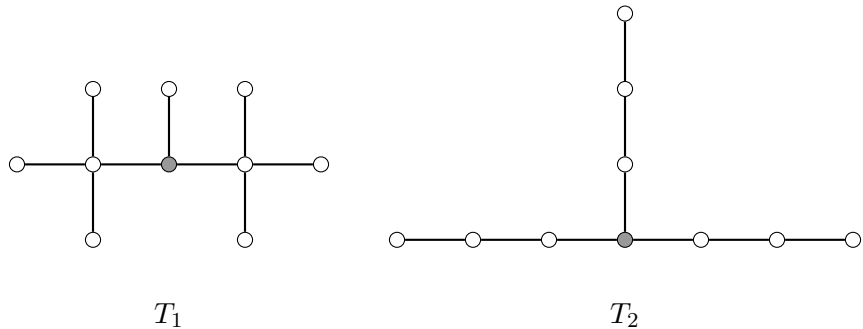

    We can use the symmetries of the rooted trees illustrated in Figure~\ref{fig:Benjamini_trees} to produce two smaller Markov chains with the same property. That is, they have equal distributions of the return time to the marked state. They are illustrated in Figure~\ref{fig: reduced bats}.
    
    Let $A = B = \{1,2,3,4\}$ and $C = \{0,1\}$. Define $f:A\to C$ and $g:B\to C$ by 
    $$f(2)=0, \quad g(1)=0,$$
    and $f(i)=g(j)=1$ in all other cases.    
    Let 
    \[
        P_X = \left(
        \begin{array}{cccc} 
            0 & 1 & 0 & 0\\
            1/3 & 0 & 2/3 & 0\\
            0 & 1/4 & 0 & 3/4\\
            0 & 0 & 1 & 0
        \end{array} \right)
        \quad\text{ and }\quad
        P_Y = \left(
        \begin{array}{cccc} 
            0 & 1 & 0 & 0\\
            1/2 & 0 & 1/2 & 0\\
            0 & 1/2 & 0 & 1/2\\
            0 & 0 & 1 & 0
        \end{array} \right)
        \,.
    \]
    The marked states in Figure~\ref{fig: reduced bats} are the states which map to $0 \in C$.
    
    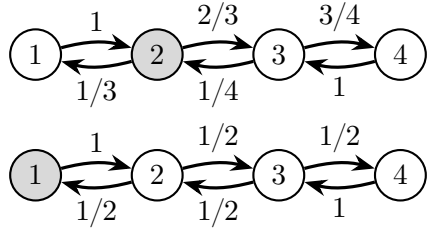
\begin{figure}
    \centering
    
    \begin{tikzpicture}[scale=0.8]
    
        \begin{scope}[every node/.style={circle,thick,draw,minimum size=0.1cm}]
    
            \node (A) at (0,2) {1} ;
            \node[fill=gray!30] (B) at (2,2) {2} ;
            \node (C) at (4,2) {3} ;
            \node (D) at (6,2) {4} ;
            
            \node[fill=gray!30] (A')  at (0,0) {1} ;
            \node (B')  at (2,0) {2} ;
            \node (C')  at (4,0) {3} ;
            \node (D')  at (6,0) {4} ;
            
        \end{scope}
        
        \begin{scope}[>={Stealth}, every edge/.style={draw=black,very thick}]
        
            \path [->] (A) edge[bend left=15] node[above] {$1$}   (B);
            \path [->] (B) edge[bend left=15] node[above] {$2/3$} (C);
            \path [->] (C) edge[bend left=15] node[above] {$3/4$} (D);
            \path [->] (D) edge[bend left=15] node[below] {$1$}   (C);
            \path [->] (C) edge[bend left=15] node[below] {$1/4$} (B);
            \path [->] (B) edge[bend left=15] node[below] {$1/3$} (A);
    
            \path [->] (A') edge[bend left=15] node[above] {$1$} (B');
            \path [->] (B') edge[bend left=15] node[above] {$1/2$} (C');
            \path [->] (C') edge[bend left=15] node[above] {$1/2$} (D');
            \path [->] (D') edge[bend left=15] node[below] {$1$}   (C');
            \path [->] (C') edge[bend left=15] node[below] {$1/2$} (B');
            \path [->] (B') edge[bend left=15] node[below] {$1/2$} (A');
            
        \end{scope}
    
    \end{tikzpicture}
    
    \caption{Chains $\bar{X}$ (above) and $\bar{Y}$ (below) in Example~\ref{ex:nonMarkovZnocoupling}.}
    \label{fig: reduced bats}
    \end{figure}
    
    Let $\upalpha_X = (0,1,0,0)$ and $\upalpha_Y = (1,0,0,0)$, so that $f(\Xoriginal_0) = g(\Yoriginal_0) = 0$, and let $\Xoriginal = \MC(\upalpha_X, P_X)$ and $\Yoriginal = \MC(\upalpha_Y, P_Y)$. We have
 
     \[f(\Xoriginal_0) = f(2) = 0 = g(1) = g(\Yoriginal_0)\]

    To check that $f(\Xoriginal) \eqdist g(\Yoriginal)$, it is enough to compute the distribution of the time $\tau_2^{\Xoriginal}$ of the first return to $2$ of $\Xoriginal$ and the time $\tau_1^{\Yoriginal}$ of the first return to $1$ of $\Yoriginal$, and see that they agree. 
    In both cases the return times are always even, and for each $k \ge 1$ we have
    \[
        \mathbb{P}(\tau_2^{\Xoriginal} = 2k) = \mathbb{P}(\tau_1^{\Yoriginal} = 2k) = \begin{cases} \frac{1}{2} & \text{ if $k=1$, }\\ \frac{2}{9} \bigl(\frac{3}{4}\bigr)^{k} & \text{if $k \ge 2$.} \end{cases}
    \]
    Note that the process $f(\Xoriginal)$ does not have the Markov property, since
    \[
        \mathbb{P}(f(\Xoriginal_4) = 0 \mid f(\Xoriginal_3) = 1) = \mathbb{P}(f(\Xoriginal_4) = 0) = \mathbb{P}(\Xoriginal_4 = 2) = 3/8
    \]
    is not equal to
    \[
        \mathbb{P}(f(\Xoriginal_4) = 0 \mid f(\Xoriginal_3) = 1, f(\Xoriginal_2) = 1) = \mathbb{P}(\Xoriginal_4 = 2 \mid \Xoriginal_2 = 4) = 1/4\,.
    \]

    We now show that there does not exist a  Markov chain, even an inhomogeneous one, $(W_t)_{t \ge 0} = (\Xcopy_t,\Ycopy_t)_{t \ge 0}$ on the state space $\Delta(f,g)$, such that $ \Xcopy \eqdist \Xoriginal$ and $ \Ycopy \eqdist \Yoriginal$.    
    Suppose (for a contradiction) that such a chain exists.  We must have $W_0 = (2,1)$. Because the chains $\Xoriginal$ and $\Yoriginal$ have period $2$, the only possible states that $W$ could visit in $\Delta(f,g)$ are the five states 
    \[
        (2,1), (1,2), (3,2), (3,4), (4,3).
    \]
    The possible transitions of $W$ form a birth-death chain, some of whose transition probabilities are forced by the marginal constraints. These constraints are illustrated in Figure~\ref{fig: possible W}. The graph is bipartite so the transition probabilities out of states whose first coordinate is even are defined only for even times $t$, and these are completely determined by the marginal constraints. The transition probabilities out of states whose first coordinate is odd are defined only for odd $t$.

\usetikzlibrary {shapes.geometric}
    \begin{figure}
    \centering
    \begin{tikzpicture}
    
        \begin{scope}[every node/.style={ellipse,thick,draw,inner sep=0.07cm}]
    
            \node (A) at (0,2) {(1,2)} ;
            \node[fill=gray!30] (B) at (2,2) {(2,1)} ;
            \node (C) at (4,2) {(3,2)} ;
            \node (D) at (6,2) {(4,3)} ;
            \node (E) at (8,2) {(3,4)} ;                
        \end{scope}
        
        \begin{scope}[>={Stealth}, every edge/.style={draw=black,very thick}]
        
            \path [->] (A) edge[bend left=15] node[above] {$1$}   (B);
            \path [->] (B) edge[bend left=15] node[above] {$2/3$} (C);
            \path [->] (C) edge[bend left=15] node[above] {$p_t$} (D);
            \path [->] (D) edge[bend left=15] node[below] {$1/2$}   (C);
            \path [->] (E) edge[bend left=15] node[below] {$1$}   (D);
            \path [->] (C) edge[bend left=15] node[below] {$1-p_t$} (B);
            \path [->] (B) edge[bend left=15] node[below] {$1/3$} (A);
            \path [->] (D) edge[bend left=15] node[above] {$1/2$}   (E);
           
        \end{scope}
    
    \end{tikzpicture}
    \caption{Transition probabilities of $W$ from time $t$ to time $t+1$ in Example~\ref{ex:nonMarkovZnocoupling}}
    \label{fig: possible W}
    \end{figure}
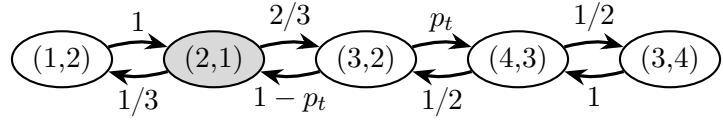
    For every odd $t \ge 1$, $\mathbb{P}(X_{t-1} = 2, X_{t} = 3) > 0$, so $\mathbb{P}(W_{t} = (3,2)) > 0$, and hence we may define the time-dependent transition probability
    \[
        p_t = \mathbb{P}(W_{t+1} = (4,3) \mid W_t = (3,2)).
    \] 
    Note that
    \[
        p_1 = \mathbb{P}(X_2 = 4 \mid X_1 = 3) = \tfrac{3}{4}.
    \]
    We will calculate $p_3$ in two ways and get distinct answers. Using $P_X$, 
    \[
        \mathbb{P}(\Xcopy_4 = 4 \mid \Xcopy_2 = 2) = \tfrac{1}{2},
    \]
    and using $W$ we get 
    \[
        \mathbb{P}(\Xcopy_4  = 4 \mid \Xcopy_2 = 2) = \mathbb{P}(W_4 = (4,3) \mid W_2 = (2,1)) =  \tfrac{2}{3}\, p_3.
    \] 
    Hence $p_3 = 3/4$.

    On the other hand, we must have $\mathbb{P}(\Xoriginal_4 = 4) = \mathbb{P}(W_4 = (4,3))$. 
    Considering $W$ we get 
    \[ 
        \mathbb{P}(W_4 = (4,3)) = \tfrac{7}{12}\, p_3   + \tfrac{1}{4} 
    \]
    and using $P_X$ we get
    \[
        \mathbb{P}(X_4 = 4) =   \tfrac{5}{8}.
    \] 
    Hence $p_3 = 11/14$. 
    This is the desired contradiction. We have shown that there is no coupling of $\Xoriginal$ and $\Yoriginal$ that takes its values in $\Delta$ and has the Markov property.

\end{example}

\begin{example}[Coupling two stationary chains]
\label{ex:stationarymarginals} This example demonstrates why Theorem~\ref{thm:stationary} is needed, by showing that the coupling provided by Theorem~\ref{thm:DTHWMC} need not be stationary even when the given chains $\Xoriginal, \Yoriginal$, and $\Zoriginal$ are stationary. 
Let $A = \{0,1\}^3 = B$, with $C = \{0,1\}$ and let $f = g$ be defined by $f(a,b,c) = b$. Let $\Xoriginal$ be the homogeneous Markov chain where the three coordinates of $\Xoriginal_0$ are independent Bernoulli($1/2$) random variables, and 
\[
    \mathbb{P}(\Xoriginal_1 = (b,c,d) \mid \Xoriginal_0 = (a,b,c)) = \tfrac{1}{2},\;\;\text{ for  $d = 0,1$.}
\]
Let $\Yoriginal$ have the same distribution as $\Xoriginal$. Then each of $f(\Xoriginal)$ and $g(\Yoriginal)$ is an i.i.d.~sequence of Bernoulli($1/2$) random variables.  Applying the construction of Theorem~\ref{thm:DTHWMC}, we obtain 
\[\Delta = \bigl\{((a,b,c),(a',b',c')) \in A \times B\,:\,b = b'\bigr\}.\]
In this example, $R^m\mathbf{1}_\Delta$  is constant for all $m \ge 1$, and we obtain
\[
    \varphi((a,b,c),(a',b,c')) = \begin{cases} 2, & \text{ if $c = c'$,}\\ 0 & \text{ if $c \neq c'$.} \end{cases}
\]
This gives a coupling whose state space $\Delta'$ is a proper subset of $\Delta$:
\[
    \Delta' = \bigl\{ ((a,b,c),(a',b,c)): a,b,c,a' \in \{0,1\}\bigr\}
\]
The transition matrix $P$ of the coupled process is given by
\[
    P\bigl(((a,b,c),(a',b,c)),((b,c,d),(b,c,d))\bigr) = \tfrac12 \quad \text{for $d=0,1$,} 
\]
with all other entries $0$.
Since $(X_0,Y_0)$ is uniformly distributed on $\Delta'$, 
\[
    \mathbb{P}(\Xcopy_0 = \Ycopy_0) = \tfrac{1}{2}
\]
but for each $t \ge 1$,  
\[
    \mathbb{P}(\Xcopy_t  = \Ycopy_t) = 1 \,.
\]
So the coupled process $(\Xcopy_t,\Ycopy_t)_{t \ge 0}$ is neither stationary nor irreducible. If we apply Theorem~\ref{thm:stationary} instead, we get a strictly smaller state space $\Delta''$ since
\[
    \varphi^\rev((a,b,c),(a',b,c')) = 
    \begin{cases} 
        2, & \text{ if $a = a'$,} \\ 
        0, & \text{ if $a \neq a'$.}
    \end{cases}
\]
This leads to the trivial coupling in which $\Xcopy_n = \Ycopy_n$ for all $n$. 

\end{example}

\begin{example}[Exact Markovian coupling but no strong Markovian coupling]
\label{ex: 3 states each affirmative}
Let $A=\{0,1,2\}$, $B=\{0',1',2'\}$, and $C=\{0,1\}$ and define the maps $f(x)=\ind{x\neq0}$ and $g(y)=\ind{y\neq0'}$ from $A$ and $B$ respectively to $C$. Define the following stochastic matrices, with rows and columns indexed by $A$, $B$, and $C$, respectively:
\begin{align*}
    P_X
    =
    &
    \begin{pmatrix}
        0 & 1/2 & 1/2 \\
        0 & 1/2 & 1/2 \\
        1 & 0 & 0
    \end{pmatrix},
    &
    P_Y
    =
    &
    \begin{pmatrix}
        0 & 1/3 & 2/3 \\
        0 & 0 & 1 \\
        3/4 & 1/4 & 0
    \end{pmatrix},
    & 
    P_Z
    =
    \begin{pmatrix}
        0 & 1 \\
        1/2 & 1/2
    \end{pmatrix}.
\end{align*}
These chains and maps are illustrated in Figure~\ref{Fig: WMC example}.
Their unique stationary probability distributions are:
\begin{align*}
    \uppi_X&=(1/3,1/3,1/3),
    & 
    \uppi_Y&=(1/3,2/9,4/9),
    & 
    \uppi_Z&=(1/3,2/3). 
\end{align*}
Let $X = \MC(\uppi_X,P_X)$, $Y = \MC(\uppi_Y, P_Y)$, and $Z = \MC(\uppi_Z,P_Z)$.
It is easy to check that $X$ lumps exactly under $f$ with respect to the family
\[ 
    \nu^f_0 = (1,0,0), \quad \nu^f_1 = (0,1/2,1/2).
\]
Likewise, $Y$ lumps exactly under $g$ with respect to the family
\[
    \nu^g_0= (1,0,0), \quad \nu^g_1 = (0, 1/3,2/3).
\]
However, $f$ and $g$ are not strong lumpings of $P_X$ and $P_Y$ to $P_Z$. 

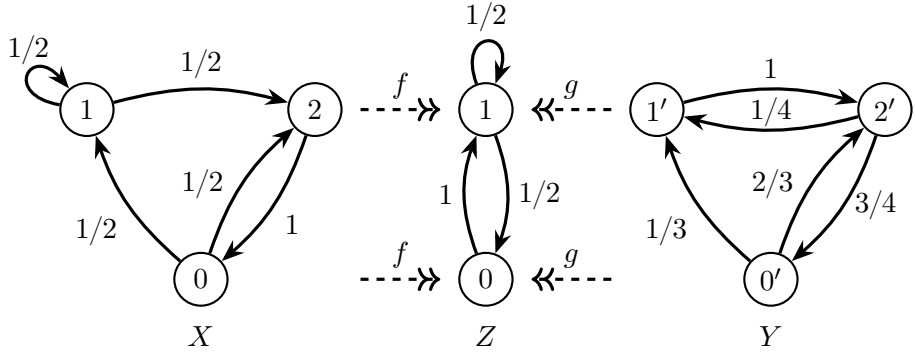
\begin{figure}
\label{fig:323}
\centering
{
\begin{tikzpicture}[scale=0.75]

    \begin{scope}[every node/.style={circle,thick,draw,minimum size=0.7cm,inner sep=0cm}]
    
        \node (A)  at (2,0)  {$0$};
        \node (B)  at (0,3)  {$1$};
        \node (C)  at (4,3)  {$2$};

        \node (1)  at (7,3)  {$1$};
        \node (0)  at (7,0)  {$0$};

        \node (A') at (12,0) {$0'$};
        \node (B') at (10,3) {$1'$};
        \node (C') at (14,3) {$2'$};
        
    \end{scope}
    
    \begin{scope}[>={Stealth}, every edge/.style={draw=black,very thick}]
    
        \path [->] (A) edge[bend left=15] node[below left] {$1/2$} (B);
        \path [->] (B) edge[out=170,in=130,looseness=8] node[above] {$1/2$} (B);
        \path [->] (B) edge[bend left=15] node[above] {$1/2$} (C);
        \path [->] (C) edge[bend left=15] node[below right] {$1$} (A);
        \path [->] (A) edge[bend left=15] node[ left=1mm] {$1/2$} (C);

        \path [->] (0) edge[bend left=20] node[left] {$1$} (1);
        \path [->] (1) edge[out=110,in=70,looseness=8] node[above] {$1/2$} (1);
        \path [->] (1) edge[bend left=20] node[right] {$1/2$} (0);

        \path [->] (A') edge[bend left=15] node[below left] {$1/3$} (B');
        \path [->] (B') edge[bend left=15] node[above] {$1$} (C');
        \path [->] (C') edge[bend left=15] node[ right] {$3/4$} (A');
        \path [->] (A') edge[bend left=15] node[ left=1mm] {$2/3$} (C');
        \path [->] (C') edge[bend left=15] node[above=-1mm] {$1/4$} (B');
        
    \end{scope}

    \begin{scope}[every edge/.style={draw=black,very thick, dashed}]
        \path [->>] (4.8,0) edge node[above] {$f$} (6.2,0);
        \path [->>] (4.8,3) edge node[above] {$f$} (6.2,3);
        \path [->>] (9.2,0) edge node[above] {$g$} (7.8,0);
        \path [->>] (9.2,3) edge node[above] {$g$} (7.8,3);
    \end{scope}

    \begin{scope}
        \node (X) at (2,-1)  {$X$};
        \node (Y) at (12,-1) {$Y$};
        \node (Z) at (7,-1)  {$Z$};
    \end{scope}

\end{tikzpicture}
}
\caption{The Markov chains and maps in Example~\ref{ex: 3 states each affirmative}.}
\label{Fig: WMC example}
\end{figure}

Any WMC of $X$ and $Y$ taking values in $\Delta$ cannot use the state $(2,1')$, because from $2$ the first marginal must step to $0$, while from $1'$ the second marginal must step to $2'$, but $(0,2') \notin \Delta$. Hence there can be no SMC  of $P_X$ and $P_Y$ that preserves $\Delta(f,g)$. 

The unique coupled probability distribution $\uppi$ supported on $\Delta \setminus \{(2,1') \}$ with the properties that $\proj^1_*(\uppi)=\uppi_X$ and $\proj^2_*(\uppi)=\uppi_Y$ is given in the following table.
\begin{center}
\begin{tabular}{c|ccccc}
$(a,b)$ & $(0,0')$ & $(1,1')$ & $(1,2')$ & $(2,1')$ &$(2,2')$\\
\hline
$\uppi(a,b)$ & $1/3$ & $2/9$ & $1/9$ & 0 & $1/3$
\end{tabular}
\end{center}
This allows us to calculate the values of $\varphi$, using 
\[
    \varphi(a,b) =   \frac{\uppi(a,b)\uppi_Z(f(a))}{\uppi_X(a)\uppi_Y(b)}
    \,
\]
as follows.
\begin{center}
\begin{tabular}{c|ccccc}
$(a,b)$ & $(0,0')$ & $(1,1')$ & $(1,2')$ & $(2,1')$ &$(2,2')$\\
\hline
$\varphi(a,b)$ & $1$ & $2$ & $1/2$ & 0 & $3/2$
\end{tabular}
\end{center}
Note that in this case, $\Delta'\neq\Delta$.

Using $\varphi$ we may compute the transition matrix $P$ described immediately after Theorem~\ref{thm:DTHWMC}, obtaining the WMC $\MC(\uppi,P)$ illustrated in Figure~\ref{fig:coupled323}. The rows and columns of $P$ are indexed by $(0,0'), (1,1'), (1,2'), (2,2')$ in order. 
\begin{figure}
\centering
\begin{tikzpicture}[scale=0.7]

    \begin{scope}[every node/.style={ellipse,thick,draw,minimum size=0.7cm,inner sep=0.06cm}]
    
        \node (AA)  at (0,0)  {$(0,0')$};
        \node (BB)  at (4,4)  {$(1,1')$};
        \node (BC)  at (0,4)  {$(1,2')$};
        \node (CC)  at (4,0)  {$(2,2')$};
        
    \end{scope}
    
    \begin{scope}[>={Stealth}, every edge/.style={draw=black,very thick}]
    
        \path [->] (AA) edge node[left=2mm] {$1/3$} (BB);
        \path [->] (AA) edge[bend left=15] node[left] {$1/6$} (BC);
        \path [->] (AA) edge[bend left=15] node[above] {$1/2$} (CC);
        \path [->] (BB) edge[bend left=15] node[below] {$1/4$} (BC);
        \path [->] (BB) edge[bend left=15] node[left] {$3/4$} (CC);
        \path [->] (BC) edge[bend left=15] node[below] {$1$} (BB);
        \path [->] (CC) edge[bend left=15] node[above] {$1$} (AA);
        
    \end{scope}

    \node at (2,-1)
    {$W$};
    
    \node at (9,2) 
    {$\displaystyle 
    P=
    \begin{pmatrix}
        0 & 1/3 & 1/6 & 1/2 \\
        0 & 0 & 1/4 & 3/4 \\
        0 & 1 & 0 & 0 \\
        1 & 0 & 0 & 0
    \end{pmatrix}
    $}; 

\end{tikzpicture}
\caption{WMC of $X$ and $Y$ in Example~\ref{ex: 3 states each affirmative}.}
\label{fig:coupled323}
\end{figure}
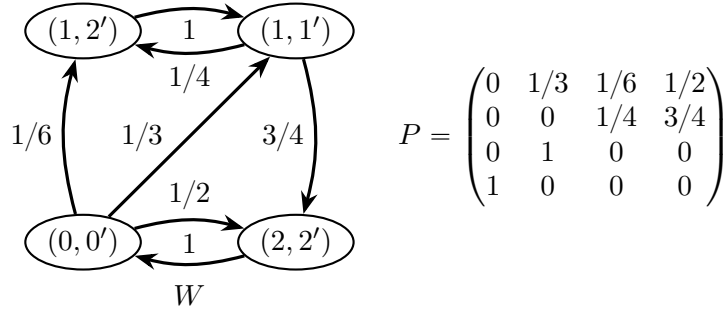

The coupling is an EMC since $\proj^1$ and $\proj^2$ are exact lumpings with respect to the following families of probability measures on $\Delta'$,
\[
    \nu^1_0 = (1,0,0,0), \quad \nu^1_1 = (0,2/3,1/3,0), \quad \nu^1_2 = (0,0,0,1),\]
\[
    \nu^2_{0'}= (1,0,0,0), \quad \nu^2_{1'} = (0,1,0,0), \quad \nu^2_{2'} = (0,0,1/4,3/4),
\]
and $\pi$ is a convex combination of each of these families.
\end{example}

\begin{example}[Coupling biased random walks on $\mathbb{Z}$]\label{ex:biasedRW}
    Let $1/2 < p < 1$, let $q = 1-p$, and let $(\Xoriginal_t)_{t \ge 0}$ be a biased simple random walk on $\mathbb{Z}$ starting at $0$, with transition probabilities
    \[
        P_X(x,x') = p\ind{x'=x+1} + q\ind{x'=x-1}.
    \]
    Perhaps surprisingly, $(|\Xoriginal_t|)_{t \ge 0}$ is also a homogeneous Markov chain. Its transition matrix is defined as follows for $z,z' \in \mathbb{N}$.
    \[
        P_Z(z,z') =
        \begin{cases}
            1, & z=0, z'=1, \\\frac{\displaystyle{p^{z+1}+q^{z+1}}}{\displaystyle{p^z+q^z}}, & z\ge1, z'=z+1, \smallskip\\
    \frac{\displaystyle{p^z q+q^z p}}{\displaystyle{p^z+q^z}}, & z\ge1, z'=z-1, \\
            0, & \text{otherwise.}
        \end{cases}
    \] 
    To verify this, for each $n\ge 0$, let $\nu_n \in \PM(\mathbb{Z})$ be given by
    \[
        \nu_n(k) = 
        \frac{p^n\ind{k=n}+q^n\ind{k=-n}}{p^n + q^n}.
    \]
    One can check that $\Xoriginal = \MC(\delta_0,P_X)$ lumps exactly to $\Zoriginal = \MC(\delta_0,P_Z)$  under the absolute value map $|\cdot|$, with respect to the family $(\nu_n)_{n\ge0}$.

    Now let $\Yoriginal$ be a copy of $\Xoriginal$. To fit the setting of Theorem~\ref{thm:DTHWMC}, take $f$ and $g$ to be the absolute value function. Clearly there is a coupling $(X,Y)_{t \ge 0}$ of $\Xoriginal$ and $\Yoriginal$ in which $X_t = Y_t$ for every $t$, and this satisfies all the conclusions of Theorem~\ref{thm:DTHWMC} apart from conditional independence of $X$ and $Y$ given $(f(X_t)_{t \ge 0})$.

    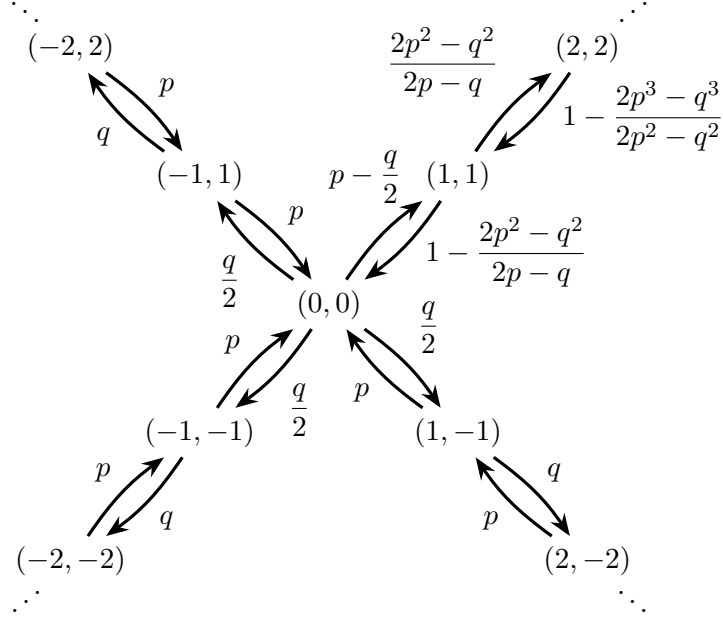
\begin{figure}
    \centering
    \begin{tikzpicture}[scale=0.85]
    
        \begin{scope}
    
            \node (A) at (0,0) {$(-2,-2)$} ;
            \node (B) at (2,2) {$(-1,-1)$} ;
            \node (C) at (4,4) {$(0,0)$} ;
            \node (D) at (6,6) {$(1,1)$} ;
            \node (E) at (8,8) {$(2,2)$} ;

            \node (F) at (0,8) {$(-2,2)$} ;
            \node (G) at (2,6) {$(-1,1)$} ;
            \node (H) at (6,2) {$(1,-1)$} ;
            \node (I) at (8,0) {$(2,-2)$} ;
            
        \end{scope}
        
        \begin{scope}[>={Stealth}, every edge/.style={draw=black,very thick}]
        
            \path [->] (A) edge[bend left=10] node[above left] {$p$}   (B);
            \path [->] (B) edge[bend left=10] node[below right] {$q$}   (A);
            \path [->] (B) edge[bend left=10] node[above left] {$p$}   (C);
            \path [->] (C) edge[bend left=10] node[below right] {$\displaystyle{\frac{q}{2}}$}   (B);
            \path [->] (C) edge[bend left=10] node[above=5mm,yshift=-2mm,xshift=-2mm] {$p-\displaystyle{\frac{q}{2}}$}   (D);
            \path [->] (D) edge[bend left=10] node[right=2mm,yshift=-1mm,xshift=-1mm] {$1-\displaystyle{\frac{2p^2-q^2}{2p-q}}$}   (C);
            \path [->] (D) edge[bend left=10] node[above left] {$\displaystyle{\frac{2p^2-q^2}{2p-q}}$}   (E);
            \path [->] (E) edge[bend left=10] node[right=2mm] {$1-\displaystyle{\frac{2p^3-q^3}{2p^2-q^2}}$}   (D);

            \path [->] (F) edge[bend left=10] node[above right] {$p$}   (G);
            \path [->] (G) edge[bend left=10] node[below left] {$q$}   (F);
            \path [->] (G) edge[bend left=10] node[above right]  {$p$} (C);
            \path [->] (C) edge[bend left=10] node[below left] {$\displaystyle{\frac{q}{2}}$}   (G);
            \path [->] (C) edge[bend left=10] node[above right] {$\displaystyle{\frac{q}{2}}$}   (H);
            \path [->] (H) edge[bend left=10] node[below left] {$p$}   (C);
            \path [->] (H) edge[bend left=10] node[above right] {$q$}   (I);
            \path [->] (I) edge[bend left=10] node[below left] {$p$}   (H);
            
        \end{scope}

        \begin{scope}
    
            \node (J) at (-0.7,8.7) {$\ddots$} ;
            \node (K) at (-0.7,-0.5) {$\iddots$} ;
            \node (L) at (8.7,-0.5) {$\ddots$} ;
            \node (M) at (8.7,8.7) {$\iddots$} ;
        \end{scope}
    
    \end{tikzpicture}
    \caption{Exact Markovian coupling of two biased simple  random walks in which the absolute values coincide.}
    \label{fig: coupled SBRWs}
    \end{figure}

    We will now exhibit the WMC
    \[
        W  = \MC(\delta_{(0,0)},P) = (X_t, Y_t)_{t \ge 0}
    \] 
    with $Z_t = |X_t| = |Y_t|$ for all $t \ge 0$, where $X$ and $Y$ are conditionally independent given $Z$.   Clearly, the initial distribution of $W$ is concentrated at $(0,0)$. To compute the transition matrix $P$, it suffices to compute $\varphi$. This may be done by computing $\mathbb{P}(W_n = (\pm n,\pm n))$ in two different ways, firstly using the explicit form of $\upalpha$ and $P$ (equations~\eqref{eq: coupled pi def} and~\eqref{eq: coupled P def}), and secondly using the decomposition of $Z$ into excursions from $0$ along with the conditional independence of $X$ and $Y$ given $Z$. We leave it as an exercise for the interested reader to verify that $\varphi(0,0) = 1$ and, for each $n \ge 1$,

         \[
             \varphi(n,n)
             =
             \frac{(p^n+q^n)(2p^n-q^n)}{2p^{2n}} 
             \,,
         \]
         \[
             \varphi(-n,n)=
             \varphi(n,-n)=
             \frac{p^n+q^n}{2p^n},
         \quad
             \varphi(-n,-n)
             =
             \frac{p^n+q^n}{2q^n}.
         \]
    This yields
    \[
        P((0,0),(1,1))=
        p-\frac{q}{2}\,,
    \]
    \[
        P((0,0),(-1,1))=
        P((0,0),(1,-1))=
        P((0,0),(-1,-1))=
        \frac{q}{2}\,,
    \]
    and for $n\ge1$, 
    \[  P((n,n),(n+1,n+1)) = \frac{2p^{n+1}-q^{n+1}}{2p^n-q^n} = 1 - P((n,n),(n-1,n-1))\,,  \]
    \begin{align*}
        P((-n,n),(-n-1,n+1)) &= q, \quad &P((-n,n),(1-n,n-1)) &=p\,,\\
        P((n,-n),(n-1,1-n)) &= q, \quad   &P((n,-n),(n-1,1-n)) &= p\,,  \\
        P((-n,-n),(-n-1,-n-1)) &= q, \quad  &P((-n,-n),(1-n,1-n) &= p\,.
    \end{align*}    
    All other entries of $P$ are $0$. The transition probabilities are illustrated in Figure~\ref{fig: coupled SBRWs}.  
    
    Proposition~\ref{cor: g exact proj1 exact} tells us that this coupling is an EMC, and the reader may enjoy verifying this directly. To check that $\proj^2$ is an exact lumping, consider the family of probability measures $(\mu_y)_{y \in \mathbb{Z}}$ on $\Delta$ given by
    \[ \mu_y = 
    \begin{cases} \frac{1}{2}\delta_{(y,y)} + \frac{1}{2}\delta_{(-y,y)} & \text{if $y < 0$,}\\
    \delta_{(0,0)} & \text{if $y = 0$,}\\[3pt]
    \displaystyle\frac{q^y}{2p^y}\,\delta_{(-y,y)} + \Bigl(1 - \frac{q^y}{2p^y}\Bigr)\delta_{(y,y)}& \text{if $y > 0$.}
    \end{cases}
    \]
    We remark that the behaviour of $W$ given $Z$ is easy to describe: for each finite excursion of $Z$ from $0$, the corresponding excursion of $W$ from $(0,0)$ picks one of the four arms of $\Delta$ uniformly, independently over excursions. The final excursion of $Z$ from $0$ that never returns to $0$ corresponds to an excursion of $W$ from $(0,0)$ up the arm where both coordinates are positive. What is not so obvious, but is shown by Theorem~\ref{thm:DTHWMC}, is that this construction makes $W$ a homogeneous Markov chain.

    Finally, we justify the assertion made in Remark~\ref{not-just-stepwise} that conditioning on all of $Z$ is in general not the same as conditioning just one step ahead. In the present example, if we tried instead to couple $\Xoriginal$ and $\Yoriginal$ by sampling $Z$ and then for each $t \ge 0$ letting $X_{t+1}$ and $Y_{t+1}$ be conditionally independent given $X_t, Y_t$, and $Z_{t+1}$, we would obtain a Markov process $(X_t,Y_t)$ whose marginals do not have the same laws as $\Xoriginal$ and $\Yoriginal$. Indeed, from $(0,0)$ the coupled chain would step to $(-1,-1)$ with probability $q^2$, and then with positive probability $Z$ would never return to $0$ and hence we would have $X_t \to -\infty$ as  $t \to \infty$, but $\Xoriginal_t \to +\infty$ a.s., so $X$ and $\Xoriginal$ cannot have the same law. 
    \enlargethispage*{1cm}
\end{example}

\section*{Acknowledgements}
The research of Edward Crane was supported by the Heilbronn Institute for Mathematical Research. Alexander E.~Holroyd was supported in part by a Royal Society Wolfson Fellowship. Erin Russell was supported by the UK Engineering and Physical Sciences Research Council, grant number EP/W52413X/1.
%
\bibliographystyle{plain}
\bibliography{bibliography}
%
\end{document}